\documentclass[a4paper, 12pt]{article}
\usepackage [a4paper,left=2.5cm,bottom=2.5cm,right=2.5cm,top=2.5cm]{geometry}
\usepackage[english]{babel}
\usepackage{amssymb}
\usepackage{amsmath,amsthm, dsfont}
\usepackage{subcaption}
\usepackage{comment}
\usepackage{tabularx,array}
\usepackage{graphicx, url}
\usepackage{enumitem} 
\usepackage{bm}
\usepackage{float}
\usepackage{bbm}
\usepackage{xcolor}
\usepackage{graphicx}
\usepackage{todonotes}
\usepackage{todonotes}
\usepackage{hyperref}
\hypersetup{colorlinks=true, linkcolor=blue,citecolor=gray}
\usepackage[nameinlink,capitalise]{cleveref}
\crefformat{equation}{#2(#1)#3}

\theoremstyle{plain}
\theoremstyle{plain}\newtheorem{assumption}{Assumption}
\crefname{assumption}{Assumption}{Assumptions}
\newtheorem{theorem}{Theorem}[section]

\newtheorem{corollary}[theorem]{Corollary}
\newtheorem{proposition}[theorem]{Proposition}

\newtheorem{remark}[theorem]{Remark}

\newcommand{\argmin}[1]{\underset{#1}
{\operatorname{arg}\!\operatorname{min}}\;}

\newcommand{\E}{\mathbb{E}}

\newcommand{\R}{\mathbb{R}}
\newcommand{\N}{\mathbb{N}}
\newcommand{\w}{\widehat}

\newcommand{\id}[1]{\ensuremath{\mathbbm{1}_{#1}}}
\renewcommand{\P}{\mathbb{P}}

\newcommand{\tr}{\text{Tr}}

\newcommand{\nlip}[1]{\| #1 \|_{\text{Lip}}}

\newcommand{\bt}{b_{\theta}}

\newcommand{\bto}{b_{\theta_0}}
\newcommand{\bth}{b_{\w\theta_L}}
\newcommand{\nop}[1]{\| #1 \|_{\text{op}}}
\newcommand{\dn}{\Delta_n}
\newcommand{\M}[2]{\mathcal{M}_{#1 \times #2}(\R)}

\title{Sampling effects on Lasso estimation of drift functions in high-dimensional diffusion processes}
\author{Chiara Amorino$^{\dagger}$, Francisco Pina,  Mark Podolskij$^{*}$}

\date{\today}
\begin{document}
\maketitle

\footnote{
The authors gratefully acknowledge financial support of ERC Consolidator Grant 815703 “STAMFORD: Statistical Methods for High Dimensional Diffusions” and  the PRIDE Grant “MATHCODA: Mathematical Tools for Complex Data Structures”.\\

{$^\dagger$ Universitat Pompeu Fabra, Barcelona (Spain) }

$^{*}$Universit\'e du Luxembourg, Esch-sur-alzette (Luxembourg)
}

\begin{abstract}
In this paper, we address high-dimensional parametric estimation of the drift function in diffusion models, specifically focusing on a $d$-dimensional ergodic diffusion process observed at discrete time points. We consider both a general linear form for the drift function and the particular case of the  Ornstein-Uhlenbeck (OU) process. Assuming sparsity of the parameter vector, we examine the statistical behavior of the Lasso estimator for the unknown parameter. Our primary contribution is the proof of an oracle inequality for the Lasso estimator, which holds on the intersection of three specific sets  defined for our analysis. We carefully control the probability of these sets, tackling the central challenge of our study. This approach allows us to derive error bounds for the $l_1$ and $l_2$ norms, assessing the performance of the proposed Lasso estimator. Our results demonstrate that, under certain conditions, the discretization error becomes negligible, enabling us to achieve the same optimal rate of convergence as if the continuous trajectory of the process were observed. We validate our theoretical findings through numerical experiments, which show that the Lasso estimator significantly outperforms the maximum likelihood estimator (MLE) in terms of support recovery. \\

 

\end{abstract}

\tableofcontents

\section{Introduction}

Over the last few decades, there has been a growing interest in statistical inference for stochastic processes, driven by the extensive range of applications of these models across various fields. For instance, stochastic processes are employed in biology \cite{Ric77}, physics \cite{Pap95} and epidemiology \cite{Bai57}. Other notable applications include neurology \cite{Hol76} and biomedical sciences \cite{Ban75}, as well as economics \cite{Ber90} and mathematical finance \cite{Hul00}. Given the importance of stochastic differential equations (SDEs), extensive research has been conducted on inference for these models. Scholars have investigated both parametric and non-parametric inference methods, utilizing either continuous or discrete observations, and within various asymptotic frameworks.

In recent years, the rise in computational power has spurred significant interest in high-dimensional statistical models. Researchers have focused on understanding statistical challenges in contexts where the number of model parameters far exceeds the number of observations. For classical models such as linear regression, issues like constructing procedures that are both computationally efficient and exhibit optimal statistical performance (quantified by convergence rates) are now well understood. However, deep statistical results for high-dimensional modeling of continuous-time processes are limited, despite their significant application potential.

High-dimensional diffusions have been extensively studied in the context of particle systems within mean field theory. Foundational probabilistic results in this area can be found in classical works such as \cite{McK66, Sni91}, while applications of particle systems across various scientific fields are explored in \cite{Bol11,Car16}, among others.
More recently, non-parametric and parametric estimation for large particle systems and their mean field limits has been investigated in studies such as \cite{Amo24, BPP, BPZ24, Pino, Nic24} and \cite{Amo23,GenLar20, Kas90, Imp21}, respectively. 

Consider a $d$-dimensional ergodic diffusion process that is the solution to the stochastic differential equation 
\begin{equation}\label{model}
    dX_t = -\bt(X_t)dt + dW_t \hspace{2cm} t\geq 0,
\end{equation}
where $\theta \in \Theta \subseteq \R^p$ is the model parameter and $\Theta$ is a compact and convex set, $\bt:\R^d \rightarrow \R^d$ is the drift function of the process and $W = (W_t)_{t \geq 0 }$  a $d-$dimensional Brownian motion. The initial value $X_{0}$ is a random vector independent of $W$.  We focus on the case where the underlying observations of the process $(X_{t_i})_{i=0,\dots,n-1}$ are discrete, where $t_j := j T/n$, $T$ is the time horizon and $\Delta_n := T/n$ the discretization step. We assume that $p, d,n$ and $T$ are large.

In the realm of high-dimensional statistics, several methods have been developed to obtain  estimation results. One prevalent approach for conducting parametric inference within the context of large-scale systems is by assuming sparsity in the parameter. In our presented work, we will assume that the true parameter of $\eqref{model}$, that we will denote by $\theta_0$, satisfies the sparsity constraint:
\begin{equation}\label{sparsityConstraint}
    \| \theta_0\|_0 := \#\{ 1 \leq i \leq p \hspace{0.2cm}| \hspace{0.2cm} \theta_{0}^i\neq 0\} = s.
\end{equation} 
The sparsity parameter is significantly smaller than the dimensionality of the parameter space. Specifically, throughout this work, we assume that the sparsity $s$ remains constant, while the parameter dimension $p$ can increase towards infinity. Furthermore, we assume that the maximum of the non-zero components is uniformly bounded with respect to 
$p$.

In the literature, a diverse array of methodologies aims to leverage the sparsity constraint of the parameter to significantly improve estimator performance. Specifically, we aim to investigate the statistical behavior of the Lasso (Least Absolute Shrinkage and Selection Operator) estimator for the unknown parameter $\theta_0$. The Lasso estimator was introduced in \cite{Tib96} and has been widely used since.

The Lasso approach is indeed a valuable and well-studied method for model selection, with the significant advantage of simultaneously performing parameter estimation and variable selection (see \cite{EHJT04,Kni00,Tib96}). It allows the dimensionality of the parameter space to change with the sample size, offering an advantage over classical information criteria such as AIC, BIC. Typically, the Lasso method involves minimizing an $l_2$ norm subject to $l_1$ norm constraints on the parameters, implying a least squares or maximum likelihood approach with additional constraints. A crucial property of the Lasso approach is the oracle property (see \cite{FaLi01}), which ensures that the correct parameters are set to zero under the true data-generating model. Originally, the Lasso procedure was introduced for linear regression problems. However, in recent years, this approach has been applied to various fields of stochastic processes. For example, in \cite{Wan07}, the problem of shrinkage estimation of regressive and autoregressive coefficients is addressed, while \cite{Nar11} studies penalized order selection for $AR(\rho)$ models. Furthermore, \cite{Bas15} demonstrates that the Bridge estimator, that is a generalization of the Lasso approach by using $l_q$ norms, can differentiate between stationarity and unit root-type non-stationarity and select the optimal lag in AR series. For additional insights into penalized estimation problems in time series analysis, see \cite{Can13}.

Recently, regularized estimators have been applied to multidimensional diffusion processes and point processes, representing a new research topic in the field of statistics for stochastic processes. In the high-frequency framework, for instance, penalized selection procedures are used to uncover the underlying true model, as discussed in \cite{Iac12,Iaf20,Mas17,Suz19}. Moreover, in \cite{Iac18}, the authors study penalized estimation for dynamical systems with small noise.

In the setting of continuous observations of an Ornstein–Uhlenbeck (OU) process and under sparsity constraints in the interaction matrix, high dimensional estimation of the drift parameter is investigated in \cite{CMP20, Gai19}. The proposed estimators are of Lasso-  and Dantzig-types. Similarly, \cite{Dex22} addresses the problem of estimating the drift parameter in a L\'evy-driven OU process.

Our paper is closely related to the work \cite{marklasso}, which investigates high-dimensional drift estimation in the model \eqref{model} under continuous time observations. This naturally raises the question of the impact of discretization on the estimation procedure and whether it is possible to achieve similar results as in the continuous case, given certain conditions on the rate at which the discretization step approaches zero.
 
To address this, {we concentrate our analysis on two primary settings. First, we investigate general linear models. Under appropriate assumptions on the drift function, this framework enables a clear identification of discretization effects and provides a deeper understanding of the role played by the various model parameters. In this context, the drift function $\bt$ is expressed as
\begin{equation}\label{linearModel}
    \bt = \phi_0 + \sum_{j=1}^{p} \theta_j \phi_j,
\end{equation}
where $\phi_j : \R^d \rightarrow \R^d$  for all $j=0,\dots,p$, and   $\theta_j > 0$ for $j=1,\dots,p.$ 
Second, we examine the case of the Ornstein–Uhlenbeck process, a canonical model in the high-dimensional inference literature due to its analytical tractability and wide range of applications. This model requires a distinct mathematical approach, primarily because of the unique behavior of concentration inequalities in the OU setting.}
In this context, we are able to prove an oracle inequality for the Lasso estimator (see Theorem \ref{TheoremOracleInequality}), which holds true on the intersection of three sets specifically tailored for our purposes. We will rigorously control the probability of these three sets in Section \ref{sectionControllingProbabilities}, addressing one of the main challenges of our paper.
We will show that the first set arises from the martingale component of the stochastic differential equation, with a continuous version detailed in \cite{marklasso}. The second set results from discretization error. We will demonstrate that the dominance of either the discretization error or the martingale component will determine which set is more significant. Specifically, we will derive a condition involving $\Delta_n$, $T$, $p$, $s$ and $d$ that must be satisfied to ensure that discretization does not affect the convergence rate. The third set pertains to the model design and includes a version of the \textit{compatibility condition}.

To prove our main results and control the probabilities of the aforementioned sets, it is essential to use concentration inequalities. As detailed in Section \ref{s: CI}, a crucial aspect of our analysis involves tracking the dependence on the step size $\Delta_n$, the dimension of the diffusion $d$, and the parameter $p$.
{While concentration bounds have been widely studied in the literature, the explicit dependence on the discretization step appears to have received little attention. This gap motivates the development of two distinct concentration inequalities: one for the general linear model, derived using martingale decomposition techniques, and another for the Ornstein–Uhlenbeck process, obtained via Malliavin calculus. These results are stated in Section \ref{s: CI}.}

Thanks to the approach summarized above, we are able to deduce error bounds for the $l_1$ and $l_2$ norms, evaluating the performance of the Lasso estimator we proposed. In particular, {for the general linear drift,} assuming that $(d\ln(p))/(n \Delta_n) \rightarrow 0$, we have that if $1/ (s^2 d n \dn^2) \rightarrow 0$, then the error term dominates the convergence of the estimator and 
  $$ \|\widehat{\theta}_L -\theta_0 \|_2 = O_{\P}\left(\frac{d }{k^2} \sqrt{s^3 \dn \ln(p)}\right).$$
 If instead $1/ (s^2 d n \dn^2) \rightarrow \infty$, then the error term is negligible and $$  \|\widehat{\theta}_L -\theta_0 \|_2 = O_{\P}\left(\frac{1}{k^2}\sqrt{\frac{ d \ln(p) s}{n \dn}}\right).$$
The constant $k$ above depends on the model and, in the case where $p \le d$, its magnitude is often $\sqrt{d}$. Therefore, when the discretization error is negligible, and recalling that $n \Delta_n = T$, we recover the same optimal rate as in the scenario where a continuous trajectory of the process was available (see \cite{marklasso}).
Our bounds are notably sharp, demonstrated by the logarithmic dependence on $p$. 

It is also worth emphasizing that our condition for determining whether the contribution of the discretization error or the martingale part is dominant is given by the asymptotic behavior of $s^2 d n \Delta_n^2$. Notably, when $d$, $s$ and $p$ are constants, the constraint for obtaining a negligible discretization error simplifies to $n \Delta_n^2 \rightarrow 0$. This condition is consistent with classical assumptions in the literature for parameter estimation of diffusion processes (see Remark \ref{rk: condition step} below for details).

{We then extend these conclusions to the Ornstein–Uhlenbeck case. As demonstrated in our analysis, the results remain essentially comparable, with the main distinction arising from the influence of the sparsity parameter. Due to the specific structure of the OU process, this parameter cannot be tracked as explicitly as in the general linear setting within our framework. Additionally, the use of a dedicated concentration inequality derived via Malliavin calculus represents one of the key methodological differences and directly impacts the resulting bounds. However, the convergence rates are equivalent across both models, with the sampling effect having the same impact in both cases.}

Moreover, to illustrate our main results, we present numerical experiments for the linear model. Specifically, we evaluate the performance of the Lasso and the MLE in a high-dimensional situation, and compare the mean $l_1$ and $l_2$ error for different values of the dimension of the true parameter. Our results show that the proposed estimator performs well. In particular, we observe that the performance of the Lasso estimator, especially in terms of support recovery, is significantly superior to that of the maximum likelihood estimator, even for a relatively small dimension of the true parameter.

{The outline of the paper is as follows. In Section \ref{ss: notation} and Section \ref{ss: ass}, we introduce the notation and the assumptions used throughout the paper. Section \ref{s: main} presents our main results, together with a detailed discussion of the methodology employed. In Section \ref{s: CI}, we establish the concentration inequalities used in our analysis, covering both the linear and the Ornstein–Uhlenbeck (OU) cases. Section \ref{s: num} validates our theoretical findings through numerical simulations. Section \ref{s: CAO} summarizes the results and outlines directions for future research. Finally, Section \ref{s: proof main} contains the proofs of the main results and of the concentration inequalities.}

\subsection{Notation}{\label{ss: notation}}
Before presenting our results, we will briefly introduce the  notation used throughout the paper.
All random variables and stochastic processes are defined on a filtered probability space \( (\Omega, \mathcal{F}, (\mathcal{F}_t)_{t \geq 0}, \P) \).
All vectors are understood as column vectors.  Given a vector $x \in \R^d$, we will denote by $x^i$ the $i$-th component of the vector and we will write $\|x\|_q$ for $q \in [1,\infty]$ to express the $l_q-$norm of $x$: 
\begin{align*}
    &\|x\|_q := \left( \sum_{i=1}^d |x^i|^q\right)^{\frac{1}{q}} \text{for   } q \in [0,\infty)
    &\|x\|_\infty := \sup_{1 \leq i \leq d} |x^i|.
\end{align*}
{
For a matrix $A \in \M{d_1}{d_2}$, we extend the $l_q$-norm entrywise, and write
\begin{align*}
&\|A\|_q := \left( \sum_{i=1}^{d_1} \sum_{j=1}^{d_2} |a_{ij}|^q \right)^{1/q} \quad \text{for } q \in [1,\infty),
&\|A\|_\infty := \max_{\substack{1 \leq i \leq d_1 \\ 1 \leq j \leq d_2}} |a_{ij}|,
\end{align*}
while  $\nop{A}$ denotes the operator norm. We denote by $A^{\star}$ its transpose and by vec$(A)$, the vector in $\R^{d_1d_2}$ obtained by stacking the columns of  A.}

For a function $f: \Theta \times \R^d \rightarrow \R^d$, we denote by $\dot{f}$ the derivative of $f$ with respect to the parameter $\theta \in \Theta$. We say a function $f: \R^d \rightarrow \R^d$ has polynomial growth when $\| f(x) \|_2 \leq C(1 + \|x\|_2^q)$ for some $q,C>0$. For a Lipschitz continuous function $f$, we  denote by $\nlip{f}$ the Lipschitz norm, and  by $C(\R^+; \R^d)$ the class of continuous functions $f:\R^+ \rightarrow \R^d$.

Given $X_{t_0}, \dots,X_{t_{n-1}}$ random vectors observed in equidistant  time intervals $t_0,\dots t_{n-1}$, we denote by $t_{i} - t_{i-1}=\dn := T/n$ the discretization step.  For any process $X$ we use the notation $\Delta X_i := X_{t_i} - X_{t_{i-1}}$. For two vectors $x,y \in \R^d$, we denote by $\langle x,y \rangle$ the scalar product. 

We introduce the function $\Phi(x): = (\phi_1(x),\dots,\phi_p(x))\in \R^{d\times p}$ whose $j$th column is the $\phi_j$ function considered in \eqref{linearModel}.
We define the Fisher information of the process by 
\begin{equation}\label{fisherInfo}
        I(\theta) = \E_{\theta}\left[ \dot{\bt}(X_{t_0})^{\star} \dot{\bt}(X_{t_0})\right].
\end{equation}
 Let us denote by $l_{\text{min}}(\theta)$ the smallest eigenvalue of  $I(\theta)$. It is worth noting that, under our hypothesis, the drift is linear, meaning its derivative does not depend on \(\theta\). Consequently, the Fisher information simplifies to \(\E_{\theta}\left[ \Phi(X_{t_0})^{\star} \Phi(X_{t_0}) \right]\). Furthermore, we will denote the Fisher information under the true value of the parameter \(\theta_0\) simply as \(I\).
 
 For a set $\mathcal{S} \subset \{1, \dots, d\}$ and a vector $x \in \R^d$ we will use the notation
\begin{equation*}
    \left(x|_{\mathcal{S}}\right)^j := x^j \id{\{j \in \mathcal{S}\}} \hspace{1.5cm} j=1,\dots,d,
\end{equation*}
{and analogously, for a matrix $A \in \M{d_1}{d_2}$, we define the notation for the restriction of the matrix $A$ to the set $\mathcal{S}$ entrywise as follows:
\begin{equation*}
    \left(A|_{\mathcal{S}}\right)_{ij} := a_{ij} \id{\{i,j\in \mathcal{S}\}} \hspace{1.5cm} i = 1,\dots, d_1,~ j = 1, \dots, d_2.
\end{equation*}
}
We define the set $\mathcal{C}(s, c)$ as
\begin{equation}\label{cone}
    \mathcal{C}(s, c) := \left\{ x \in \R^d \hspace{0.1cm}\backslash \hspace{0.1cm} \{0\} : \|x\|_1 \leq (1+c)\| x|_{\mathcal{I}_s(x)}\|_1\right\},
\end{equation}
where $\mathcal{I}_s(x)$ is the set of the $s$ largest elements of $x$, {and its definition can be straightforwardly extended to matrices.} For an element $y \in \mathbb{C}$, we denote by Re($y$) its real part.

\subsection{Assumptions}{\label{ss: ass}}
 In this section, we introduce the main assumptions on the model $\eqref{model}$. 
 \subsubsection{Assumptions for the general linear drift}
 {As previously mentioned, we first present the results related to the general linear drift, that is, when the drift function in \eqref{model} takes the form given in \eqref{linearModel}. In this context}, we assume the following:
 
 \begin{assumption}\label{A1} 
 For all $j=0,\dots,p$, the functions $\phi_j: \R^d \rightarrow \R^d$ are uniformly Lipschitz in $p$ and $d$, and we denote by  $L_j:=\|\phi_j\|_{\text{\rm Lip}}$ their Lipschitz norms.    
 \end{assumption}

 \begin{assumption}\label{A2} There exists a constant $M > 0$ such that for all $x,y \in \R^d$ and $\theta \in \Theta$, we have that 
\begin{equation*}
\left\langle b_{\theta}(x) - b_{\theta}(y),x-y\right\rangle \geq M\|x-y\|_2^2.
\end{equation*}
\end{assumption}

\noindent
\cref{A1} and \cref{A2} imply that equation $\eqref{model}$ admits a unique strong solution $X_t$, which is a homogeneous continuous Markov process with an invariant distribution (cf. \cite{Rogers_Williams_2000}, Theorem 12.1). This solution has bounded moments of any order. Henceforth, we will assume that these moments are uniformly bounded in all quantities of interest.

\begin{assumption}\label{A4}   For all $i,j=1,\dots,p$ the functions $\left\langle \phi_i , \phi_j\right\rangle:\R^d \rightarrow \R$ are assumed to be Lipschitz continuous, with constants uniformly bounded in $p$ and $d$, and we denote  $L_{ij}:=\|\left\langle \phi_i , \phi_j\right\rangle\|_{\text{\rm Lip}}$.
\end{assumption}

\begin{remark}\label{remarkBoundedUnif} 
\normalfont
\cref{A1} implies that the drift function \( b_{\theta}:\R^d \longrightarrow \R^d \) of the diffusion process is Lipschitz, and we denote its Lipschitz norm by \( L \). Notably, \cref{A1} implies that \( L \leq (L_0 + \sum_{j = 1}^p \theta_j L_j) \), indicating that the drift function is not uniformly Lipschitz in the variables of interest. However, we will utilize the Lipschitz property of \( b \) only in the proof of the concentration inequality in Section \ref{s: CI}, where the drift is evaluated at the true parameter value \(\theta_0\). In this context, \( b \) has a Lipschitz constant \( L_0 + s \max_k \theta_0^k L_k \), which is uniformly bounded in $p$ and $d$, under our assumptions.
 Moreover, since $\bt$ is of the form \eqref{linearModel}, the drift function is continuously differentiable in  $\theta$ and  $\dot{\bt}(\cdot)$ has polynomial growth. We also point out that, in terms of the drift function, \cref{A4} implies that the functions $( \dot{\bt}^{\star}\dot{\bt})_{i,j}: \R^d \rightarrow \R$ are Lipschitz continuous. 
 \end{remark}

\begin{remark}\label{remarkAssumptionLipschitz}
\normalfont
It is important to note that the assumption of Lipschitz continuity for the functions $\langle\phi_i, \phi_j \rangle$ does not hold in the case of the OU process. \cref{A4} plays a crucial role in deriving the concentration inequality that we employ in the general linear model. This significant difference motivates us to treat the OU process separately, leading to the use of alternative techniques, such as Malliavin calculus, to derive the corresponding concentration inequality.
\end{remark}

\begin{assumption}\label{A5} Given $\theta \in \R^p$ such that $\|\theta\|_0 = s$ and $\vartheta \in \R^p $ satisfying  $\theta - \vartheta \in \mathcal{C}(s, 3 + 4/\gamma)$ (defined in \eqref{cone}), for an arbitrary $\gamma >0$, there exists a constant $l >0 $ such that 
$$\frac{(\theta - \vartheta)^{\star}I(\theta - \vartheta)}{\|\theta - \vartheta\|_2^2} > l.$$
\end{assumption}
 
 \begin{remark}\label{remark3}
 \normalfont
\cref{A5} can be seen as a relaxation of the strong convexity condition of the Fisher information, particularly when the drift function takes the form given by \eqref{linearModel}. This assumption draws inspiration from similar conditions discussed in \cite{sara11,tibshi}. Our motivation stems from the need to provide results applicable to real high-dimensional settings where $p > d$.

Similar assumptions have also been considered in \cite{marklasso}. However, by leveraging the independence of the Fisher information with respect to the parameter and the geometric property of the Lasso error—which asserts that the estimation error lies within a subset of $\mathbb{R}^p$ as defined in \eqref{cone} (see \cite{tibshi}, Chapter 10 for further details)—we can formulate our assumption within a more general framework that captures the high-dimensional landscape. Specifically, \cref{A5} avoids imposing a condition of the type $l_{\min}(\theta) > 0$ for all $\theta \in \Theta$, which may not hold when the information matrix is degenerate. Instead, it replaces this with an identifiability condition on $I$ for a specific subspace in $\mathbb{R}^p$. 
\end{remark}

\subsubsection{{Assumptions for the OU case}}\label{sss: assumptionsOU} 
We now turn our attention to the Ornstein–Uhlenbeck process. Specifically, we consider a \( d \)-dimensional process defined as the solution to the stochastic differential equation  
\begin{equation}\label{OUdefinition}
    dX_t = -A_{0}X_t\,dt + dW_t,  \hspace{1cm} t \geq 0.
\end{equation}  
 Here, \( A_{0} \in \mathbb{A} \subseteq \mathbb{M}_{d \times d} \) denotes the interaction matrix of the process, and \( W \) is a \( d \)-dimensional Brownian motion.
In this setting, we make the following assumptions:
 \begin{assumption}\label{AOU} 
The interaction matrix \( A_0 \) is assumed to be diagonalizable, with eigenvalues \( \tau_1, \dots, \tau_d \in \mathbb{C} \). Moreover, all eigenvalues have strictly positive real parts, and we define  
\[
\mathfrak{m} := \min_{1 \leq i \leq d} \text{Re}(\tau_i) > 0.
\]  
 It implies that  
 \begin{equation}
     A_0 = P_0 \text{diag}(\tau_1, \dots, \tau_d){P_0}^{-1}, 
 \end{equation}
 where the column vectors of $P_0$ are the eigenvectors of $A_0$. Moreover, we denote by $\mathfrak{p}_0:= \|P_0\|_{\text{op}}\|P_0^{
 -1
 }\|_{\text{op}}$.
 \end{assumption}
\noindent
Under \cref{AOU}, the stochastic differential  equation \eqref{OUdefinition} exhibits a unique stationary solution, which can be written explicitly as 
\begin{equation*}
    X_t = \int_{-\infty}^t\exp(-(t-s)A_0)dW_t.
\end{equation*}
In this situation, it is well known that 
\begin{equation*}
    X_t \sim \mathcal{N}(0, C_{\infty}), \hspace{1cm } C_{\infty} := \int_0^\infty \exp{(-sA_0)}\exp{(-sA_0^{\star})}ds.
\end{equation*}
\begin{assumption}\label{LminLmaxCinf}
The  covariance matrix \( C_{\infty} \) is assumed to be non-degenerate. We let \( 0 < \mathfrak{l}_{\min} < \mathfrak{l}_{\max} < \infty \) denote its minimum and maximum eigenvalues, respectively.
\end{assumption}

\begin{remark}
\normalfont
Since the OU process has a linear drift, the interaction matrix \(A_0\) 
can be expressed in the general form~\eqref{linearModel}. In particular, by taking \(p = d^2\) 
and choosing the functions
\begin{equation}\label{phiforOU}
\phi_{ij}(x) = E_{ij} x, \quad 1 \leq i,j \leq d,
\end{equation}
where \(E_{ij}\) denotes the \(d \times d\) matrix with a single one at entry \((i,j)\) 
and zeros elsewhere, enumerating the pairs \((i,j)\) row-wise 
yields the family \(\{\phi_i\}_{i=1}^{d^2}\). The drift can then be written as
\[
A_0 x = \sum_{i=1}^d \sum_{j=1}^d a_{ij} \, \phi_{ij}(x),
\]
where \(a_{ij}\) denotes the \((i,j)\)-th entry of \(A_0\), 
playing a role analogous to the coefficients \(\theta_i\) in~\eqref{linearModel}. 
\\
Considering the functions \(\phi_{ij}\) as in~\eqref{phiforOU}, we can draw a parallel between the assumptions of the two cases. First, we observe that for the OU case, \cref{A1} 
is directly satisfied with \(L_{ij} = 1\). On the other hand, both \cref{A2} and \cref{AOU} capture the same stabilizing effect of the drift.
While \cref{A2} expresses this through a monotonicity condition, 
\cref{AOU} formulates it in spectral terms, providing a more general condition that includes a wider class of interaction matrices 
and aligns naturally with the classical OU literature. Finally, we note that  \cref{A5} and \cref{LminLmaxCinf} are also related, both reflecting the idea of identifiability. In this context, \cref{A5} provides a more general formulation for linear models, as explained in Remark \ref{remark3}, 
while \cref{LminLmaxCinf} aligns with the standard assumptions used in the OU literature. Finally, unlike in the general linear model, the OU process does not satisfy \cref{A4} (see Remark \ref{remarkAssumptionLipschitz}). Therefore, these differences motivate us to present the results for the two models separately, in order to clearly convey and facilitate their understanding.
\end{remark}
\noindent
As in the linear case,  we consider the setting where the process  is observed at discrete time points $(X_{t_i})_{i=0,\dots,n-1}$, with $t_j := j T/n$,  where $T$ denotes the time horizon and $\Delta_n := T/n$ the discretization step. Directly related with the observations of the process, we define the empirical covariance matrix of the OU as 
\begin{equation}
    C_{T} := \frac{1}{n}\sum_{i=1}^n X_{t_{i-1}}X_{t_{i-1}}^{\star},
\end{equation}
which serves as the analogue of the Gram matrix in linear regression models and constitutes a key quantity in the subsequent analysis.
\section{Main results}{\label{s: main}}
In this section, we present and discuss the main results of our paper, starting with the general linear model and subsequently addressing the OU case.

\subsection{Main results for the general linear case}
{We begin by considering the case in which the drift function satisfies the general linear model \eqref{linearModel}. Our focus lies in the high-dimensional regime, where \( p \), \( d \), \( n \), and \( T \) are all large, and we assume that the unknown parameter \( \theta_0 \) satisfies the sparsity constraint \eqref{sparsityConstraint}.
In the idealized setting where the full trajectory of the process \( (X_t)_{t \geq 0} \) is observed, Girsanov’s theorem enables an explicit expression for the scaled negative log-likelihood of the drift parameter in~\eqref{model}, given by  
\begin{equation}\label{likelihood_Continuous}
    L_T(\theta) = \frac{1}{T} \int_0^T b_\theta^{\star}(X_t) \, dX_t + \frac{1}{2T} \int_0^T \|b_\theta(X_t)\|^2 \, dt.
\end{equation}
However, our work is concerned with the more realistic scenario where the process is observed only at discrete time points. In such cases, the likelihood function cannot be directly employed to estimate the unknown parameter \( \theta_0 \). Various methods have been developed to approximate the likelihood in diffusion models (see, e.g., \cite{Yos92}). We propose the following contrast function:
\begin{equation}\label{contrastFunction}
    R_T(\theta):= \frac{1}{T} \sum_{i=1}^{n} \left\| \Delta X_i + \Delta_n \, b_\theta(X_{t_{i-1}})\right\|_2^2,
\end{equation}
where the norm denotes the standard Euclidean norm in \( \mathbb{R}^d \).
This contrast function is motivated by a first-order approximation of the continuous-time likelihood~\eqref{likelihood_Continuous}. It offers several advantages: it is simple to implement, computationally efficient, and its structure yields a natural and interpretable estimator, making it well-suited for a wide range of practical applications.}

Since our aim is to understand the statistical behaviour of the Lasso estimator, we define the discrete Lasso estimator as the penalised contrast function:
\begin{equation*}
        \w{\theta}_L := \argmin{\theta \in \Theta}\left(R_T(\theta) + \lambda \|\theta\|_1\right),
\end{equation*}
where $\lambda > 0$ is a tuning parameter. Our first goal is, using the properties of the Lasso estimator, to obtain an inequality that provides the basis for deriving an oracle inequality. For this purpose, we define the following bi-linear random form. Given $f,g$ two smooth functions such that $f,g : \R^d \rightarrow \R^d$, we define 
\begin{equation}\label{bilinearForm}
    \langle f,g \rangle_D = \frac{1}{n} \sum_{i=1}^{n} f(X_{t_{i-1}})^{\star} g(X_{t_{i-1}})
\end{equation}
 and set $\| f \|_D:= \sqrt{\langle f,f \rangle_D}$ the associated norm.

Given $n$ observations of $\eqref{model}$, and considering that $\theta_0 \in \Theta$ is the true parameter of the diffusion drift, for any  $\theta_1, \theta_2,  \in \Theta$ we introduce two random functions  
\begin{equation}\label{rf1}
     A(\theta_1,\theta_2) := \frac{1}{n} \sum_{i=1}^{n} \left(b_{\theta_1}(X_{t_{i-1}}) - b_{\theta_2}(X_{t_{i-1}}) \right)^{\star}  \left( \int_{t_{i-1}}^{t_i} \left(\bto(X_{t_{i-1}}) - \bto(X_s) \right)  ds \right),
\end{equation}
\begin{equation}\label{rf2}
     G(\theta_1,\theta_2) := \frac{1}{n} \sum_{i=1}^{n} \left( b_{\theta_1}(X_{t_{i-1}}) - b_{\theta_2}(X_{t_{i-1}}) \right)^{\star} \left(\Delta W_i \right).
\end{equation}      
The random function at \eqref{rf1} will  play an essential role when evaluating the error associated to the discretization of the process. On the other hand,  $\eqref{rf2}$ can be seen as a martingale capturing the empirical behaviour of the increments of the underlying diffusion.

In this work, we prove  an oracle inequality that holds true with high probability. To accomplish this, we will need to ensure some control over the probabilities of the following sets:
\begin{equation*}
    \mathcal{T} := \left\{ \left\|\sup_{\theta' \in \Theta} \frac{1}{n} \sum_{i=1}^{n} \dot{b_{\theta'}}(X_{t_{i-1}})^{\star}\Delta W_i \right\|_{\infty} \leq \frac{\lambda}{4}\right\},
\end{equation*}
\begin{equation}\label{settaup}
    \mathcal{T}' := \left\{ \left\| \sup_{\theta' \in \Theta} \frac{1}{n} \sum_{i=1}^{n} \dot{b_{\theta'}}(X_{t_{i-1}})^{\star} \left(\int_{t_{i-1}}^{t_i} \left( \bto(X_s) - \bto(X_{t_{i-1}})\right)ds\right)\right\|_{\infty} \leq \frac{\lambda}{4}\right\},
\end{equation}
\begin{equation*}
    \mathcal{T}'' := \left\{ \inf_{\substack{\theta \in \R^p  : \| \theta \|_0 =s \\ \upsilon \in \R^p : \theta - \upsilon \in \mathcal{C}(s,3 + 4/\gamma)}}{\frac{\|\bt - b_\upsilon\|_D}{\|\theta-\upsilon\|_2} \geq k}\right\},
\end{equation*}
for a constant $k>0$ that will be specified later,  while the parameter $\gamma>0$ appearing in $\mathcal{T}''$ remains arbitrary.

The first set, $\mathcal{T}$, is related to the stochastic behavior of the process, and we will refer to it as the empirical or martingale set. To provide an intuitive understanding, we can think that, under $\mathcal{T}$, the stochastic behavior of the diffusion process is controlled or "overruled" by the penalty parameter $\lambda$ of the Lasso estimator.

A key difference from most results in the literature (e.g., \cite{CMP20,marklasso, Dex22}) is that the trajectories of the process are not fully observed, introducing additional randomness to the model design of our estimator. Therefore, the set $\mathcal{T}'$ is associated with the discretization error induced by approximating the log-likelihood function in the continuous case, adding a new challenge in studying the properties of the Lasso estimator.

The final set, $\mathcal{T}''$, is related to the model design. As is common in Lasso literature, we can see the set $\mathcal{T}''$ as incorporating a version of the so-called compatibility condition, restricted to the subset of $\mathbb{R}^p$ that satisfies the properties of the error provided by the Lasso estimator. Furthermore, the set $\mathcal{T}''$ establishes a connection between the Euclidean norm $\| \cdot \|_2$ and the empirical norm $\|\cdot \|_D$ based on the observations of the process. This relationship is controlled by the constant $k$, which must satisfy $0 < k < \sqrt{l}$, where $l$ is the constant appearing in \cref{A5}. This connection between norms is crucial for deriving the oracle inequality presented in the following theorem.

\begin{theorem}\label{TheoremOracleInequality}
Assume that $\|\theta\|_0 = s$ and the drift function $\bt$ in \eqref{model} is differentiable with respect to $\theta$. Therefore,  for any $\gamma>0$, on $\mathcal{T} \cap \mathcal{T'} \cap \mathcal{T''}$, it holds that
 \begin{equation}\label{OracleIneq}
        \|\bth - \bto\|_D^2  \leq  (1+\gamma)\|\bt - \bto\|_D^2 + \frac{4\lambda^2 s (2 + \gamma)^2}{k^2 \gamma \dn^2},
    \end{equation}
    where $\lambda, s$ and $k$ are the tuning parameter of the Lasso estimator, the sparsity of $\theta_0$ and the constant on $\mathcal{T}''$, respectively.
\end{theorem}
\noindent
The proof  can be consulted in Section \ref{ProofOracleInequality}. 
\noindent
The oracle inequality \eqref{OracleIneq} for the Lasso estimator serves as a starting point to evaluate the performance of our estimator. In particular, we derive error bounds for the $l_1$ and the $l_2$ norms.  For this purpose, for a fixed value of $\epsilon \in (0,1)$, we define the constants
\begin{align*} 
    &\lambda_{1,1}:= 23 \sqrt{\frac{d R \dn}{n} \left[\ln{(2p)} +\ln{\left( \frac{2}{\epsilon}\right)} \right]},\\
    &\lambda_{1,2}:=7 \sqrt[4]{\frac{d^3 H_{\dn} \dn}{n^3} \left[\ln{(2p)} +\ln{\left(\frac{2}{\epsilon}\right)} \right]^{3}},\\
     &\lambda_1 := \max \left\{ \lambda_{1,1}, \lambda_{1,2}\right\},\\
      &\lambda_2 := 8 e C_b^{1/2}  s d \dn^{\frac{3}{2}}\sqrt{\ln(p) + \ln\left(\frac{1}{\epsilon}\right)}, \\
    &T_1 := \frac{324 d^2 K_{\dn} \left(5 + \frac{4}{\gamma}\right)^4}{(l -k^2)^2}\left[\ln{\left(\frac{2}{\epsilon}\right)} + \ln{\left(21^{2s}\left(p^{2s} \wedge \left(\frac{ep}{2s}\right)^{2s}\right)\right)}\right].
\end{align*} 
Constants $\lambda_1, \lambda_2$ and $T_1$ are directly related to the  sets $\mathcal{T}, \mathcal{T}'$ and $\mathcal{T}''$, respectively.   The quantities  $H_{\dn}$, $R$ , and $K_{\dn}$, whose values are uniformly bounded, are closely connected with the properties of the drift function in $\eqref{model}$ and are defined in \eqref{quantitiesSetT}, \eqref{R} and \eqref{K}, respectively. The value of $C_b$ in $\lambda_2$ is a constant depending on the model as well. It is independent of $p,d,\dn$ and $T$ and its introduction can be found in Proposition \ref{propTapprox}. In Section   $\ref{sectionControllingProbabilities}$, we will provide a detailed explanation of these expressions. 

The following theorem provides a control over the probability of the intersection of $\mathcal{T}, \mathcal{T}'$ and $\mathcal{T}''$ in terms of the defined quantities.
\begin{theorem}\label{TheoremUnionTaus}
    Under \cref{A1}-\cref{A5}, for a fixed vale of $\epsilon \in \left(0,\frac{1}{3}\right)$,  for all $\lambda > \max\{\lambda_1, \lambda_2\}$, $T > T_1$, it holds that 
    \begin{equation*}
        \P_{\theta_0}(\mathcal{T} \cap \mathcal{T}' \cap \mathcal{T}'' )\geq 1-3\epsilon.
    \end{equation*}
\end{theorem}
\begin{proof}
    The result comes directly from combining Theorem \ref{TheoremT}, Theorem \ref{TheoremTP} and  Theorem \ref{TheoremTPP} in Section \ref{sectionControllingProbabilities}.
\end{proof}

\noindent
An immediate  consequence of Theorem \ref{TheoremUnionTaus} is the following result.

\begin{corollary}\label{CoroErrorBounds}
    Assume conditions \cref{A1}-\cref{A5}. Recall that  $\|\theta_0\|_0 = s$. For a fixed value of $\epsilon \in (0,1/3)$, if $\lambda > \max\{\lambda_1,\lambda_2\}$, $0 < k <\sqrt{l}$ and  $T > T_1$, with probability at least $1-3\epsilon$, we have that 
    \begin{equation*}
         \|\widehat{\theta}_L -\theta_0 \|_2^2 \leq \frac{4\lambda^2 s (2 + \gamma)^2}{k^4 \gamma \dn^2}
    \end{equation*}
    and
    \begin{equation*}
     \|\widehat{\theta}_L -\theta_0 \|_1 \leq  \frac{8\lambda s (1+\gamma)(2+\gamma)}{k^2  \gamma^{\frac{3}{2}}\dn}.
    \end{equation*}
\end{corollary}
\begin{proof}
First, we observe that by Theorem \ref{TheoremUnionTaus}, we are on $\mathcal{T} \cap \mathcal{T}' \cap \mathcal{T}''$ with probability at least $1-3\epsilon$. Therefore, if we apply the oracle inequality $\eqref{OracleIneq}$ to $\theta = \theta_0$, we can easily conclude that 
\begin{equation*}
     \|\bth - \bto\|_D^2  \leq   \frac{4\lambda^2 s (2 + \gamma)^2}{k^2 \gamma \dn^2}.
\end{equation*}
 Under $\mathcal{T}''$ we have that $\|\bth - \bto\|_D  \geq k \|\widehat{\theta}_L -\theta_0 \|_2$, and then 
\begin{equation}\label{boundl2}
    \|\widehat{\theta}_L -\theta_0 \|_2^2 \leq \frac{4\lambda^2 s (2 + \gamma)^2}{k^4 \gamma \dn^2}.
\end{equation}
For the $l_1$ bound, as $\widehat{\theta}_L -\theta_0 \in \mathcal{C}(s,3 + 4/\gamma)$ (defined in $\eqref{cone}$), and applying Cauchy--Schwarz inequality, 
\begin{equation*}
    \|\widehat{\theta}_L -\theta_0 \|_1 \leq \left(4 + \frac{4}{\gamma}\right)\|\w\theta_L|_{\mathcal{S(\theta)}} - \theta_0\|_1 \leq \frac{4(\gamma +1)}{\gamma}\sqrt{s}\|\widehat{\theta}_L -\theta_0 \|_2 \leq \frac{8\lambda s (1+\gamma)(2+\gamma)}{k^2 \gamma^{\frac{3}{2}} \dn}.
\end{equation*}
\end{proof}

\noindent
One of our main findings is the rate of convergence of the estimator in terms of the model parameters. From Corollary \ref{CoroErrorBounds}, we can specify precise conditions under which the error term associated with the discretization of the process dominates the convergence rate of the Lasso estimator. This highlights scenarios where the lack of continuous observation of the process leads to a significant difference compared to cases where the entire path is observed.

\begin{corollary}\label{consistencyResult} Assume conditions \cref{A1}-\cref{A5}. Let us consider that  $\|\theta_0\|_0 = s$. For a fixed value of $\epsilon \in (0,1/3)$, if $\lambda > \max\{\lambda_1,\lambda_2\}$, $0 < k <\sqrt{l}$ and  $T > T_1$, with probability at least $1-3\epsilon$ we have that 
\begin{itemize}
    \item[(i)] If $(d\ln(p))/(n\dn ) \rightarrow 0$ and $1/( s^2 d n \dn^2 ) \rightarrow 0$, then the error term dominates the convergence of the estimator and 
   \begin{equation}\label{rateconvl2}
        \|\widehat{\theta}_L -\theta_0 \|_2 = O_{\P}\left(\frac{d}{k^2} \sqrt{s^3 \dn\ln(p)}\right).
   \end{equation}
   \item[(ii)] If $(d\ln(p))/(n\dn ) \rightarrow 0$ and $1/( s^2 d n \dn^2 ) \rightarrow \infty$, the error term is negligible and 
   \begin{equation}\label{rateconvl22}
        \|\widehat{\theta}_L -\theta_0 \|_2 = O_{\P}\left(\frac{1}{k^2}\sqrt{\frac{ d \ln(p) s}{n \dn}}\right).
   \end{equation}
   \end{itemize}
\end{corollary}
\begin{proof}
    Under the mentioned assumptions,  if $(d\ln(p))/( n\dn) \rightarrow 0$ and $1/( s^2 d n \dn^2) \rightarrow 0$, we have that $\max\{\lambda_1,\lambda_2\} = \lambda_2.$ Therefore, from $\eqref{boundl2}$, we obtain $\eqref{rateconvl2}$.
    On the other hand, when $(d\ln(p))/ (n \dn) \rightarrow 0$ and $ 1/ (s^2 d n \dn^2) \rightarrow \infty$, the maximum between $\lambda_1$ and $\lambda_2$ is given by $\lambda_{1,1}$. Therefore,  the term associated with the empirical part of the process is the one who dominates and, from $\eqref{boundl2}$, we obtain $\eqref{rateconvl22}$. 
\end{proof}
\begin{remark}{\label{rk: condition step}} 
\normalfont
As mentioned in the introduction, Corollary \ref{consistencyResult} indicates that the threshold determining whether the discretization error or the martingale part is dominant is given by the asymptotic behavior of $ 1/ s^2 d n \Delta_n^2$. When $d, s$ and $p$ are constants, this simplifies to $n \Delta_n^2 \rightarrow 0$. This assumption aligns with the condition found in the literature for parameter estimation in classical SDEs, as initially introduced by \cite{FloZmi}. This condition was later improved to $n \Delta_n^3 \rightarrow 0$ by \cite{Yos92}, thanks to a correction in the contrast function. Finally, Kessler \cite{Kes} proposed a contrast function based on a Gaussian approximation of the transition density, which allowed for a weaker condition $n \Delta_n^\gamma \rightarrow 0$ for any arbitrary integer $\gamma$. Similar developments have been made for classical SDEs with jumps in \cite{  Amo20, Amo21, GLM, Shi}.

One might wonder if, in our context, it would be possible to relax the condition on the discretization step similarly.
To begin with, it is important to highlight that the analysis of discretization error, as presented in Proposition \ref{propTapprox}, can be significantly simplified when the drift is bounded. In such cases, applying the Itô formula leads us to the condition $1 / (s^2 d n \Delta_n^3) \rightarrow \infty$ to ensure that the martingale part is the main contribution, and the error due to discretization becomes negligible. Even with an unbounded drift, we believe that utilizing higher-order approximations and iterating the Itô formula (similar to the methods proposed in \cite{Amo20} and \cite{Kes}) could yield a less stringent condition. However, this would introduce additional stochastic terms into our analysis, which would need to be controlled using concentration inequalities (see Section \ref{s: CI}), with explicit dependence on the variables of interest. This adjustment would significantly alter our approach, resulting in more complex computations and interpretations. Consequently, we have left this for future exploration. 
\end{remark}
\begin{remark} 
\normalfont
 If $p \leq  d$, meaning that the dimension of the stochastic process is greater than or equal to the dimension of the parameter, it is reasonable to add a classical assumption about the minimum eigenvalue $l_{\min}(\theta)$ of the empirical Fisher information matrix \eqref{fisherInfo}. Specifically, if we assume $l_{\min}(\theta) > 0$, then by choosing $k = \sqrt{l_{\min}(\theta)}/2$ and remarking that the magnitude of $l_{\min}(\theta)$ is usually of order $d$, we found that \eqref{rateconvl22} could be expressed as 
    \begin{equation*}
        \|\widehat{\theta}_L -\theta_0 \|_2 = O_{\P}\left(\sqrt{\frac{ \ln(p) s}{ d n \dn}}\right).
    \end{equation*}
Therefore, in the case where the discretization error is negligible, we achieve the same optimal rate of convergence for the estimator as when the entire path of the process is observed (see \cite{marklasso}). 
\end{remark}

\subsection{{Main results for the OU case}}
As in the general linear model, we focus on the high-dimensional regime where \( d \), \( n \), and \( T \) are large. In the Ornstein–Uhlenbeck (OU) setting, the parameter of interest is the interaction matrix \( A_0 \in \mathbb{A} \subseteq \mathbb{M}_{d \times d} \), which plays the same role as the vector \( \theta_0 \in \mathbb{R}^p \) in the linear model. the subspace $\mathbb{A}$ is also considered to be convex and compact as $\Theta$.

For the sake of clarity and conciseness throughout the rest of the paper, we emphasize the parallels between the two settings by occasionally identifying the matrix \( A_0 \) with the parameter vector \( \theta_0 \). Under this identification, the drift function becomes \( b_\theta(x) = Ax \), and the dimension of the parameter space becomes \( p = d^2 \). Consequently, when referring to results analogous to those established in the general linear setting, we implicitly interpret $\theta$ as corresponding to the matrix $A$ in the OU context.

For the estimation of the matrix \( A_0 \), we again consider a contrast function of the form given in \eqref{contrastFunction}. The Lasso estimator for the interaction matrix is then defined as:
\[
\widehat{A}_L := \arg \min_{A \in \mathbb{A}} \left\{ R_T(A) + \lambda \|A\|_1 \right\},
\]
where \( \lambda > 0 \) is the tuning parameter of the estimator.

Given two matrices $A,B \in \M{d}{d}$, the bi-linear random form defined in \eqref{bilinearForm} will read as 
\begin{equation}\label{bilinearFormOU}
    \langle A,B \rangle_D = \frac{1}{n}\sum_{i=1}^n (AX_{t_{i-1}})^\star BX_{t_{i-1}},
\end{equation}
while the random functions \eqref{rf1} and \eqref{rf2} remain analogous, as do the three sets $\mathcal{T}$, $\mathcal{T}'$ and $\mathcal{T}''$.

Our first relevant result is the oracle inequality, stated in Theorem  \ref{TheoremOracleInequality}. Since we consider that $\|A_0\|_0 = s$ and the drift function in the OU case is differentiable with respect to $A$, the inequality remains true for this particular model too.
Similarly, for a fixed value of $\epsilon \in (0,1)$, we define the following constants
\begin{align*}
     &\lambda_1^{OU} := \sqrt{\frac{32\dn(\mathfrak{a} + \mathfrak{l}_{\min}-k^2)\left[\ln(d^2) + \ln(2/\epsilon)\right]}{n}},\\
     &\lambda_2^{OU} := 8 e (C_b^{OU})^{1/2} d^2 \dn^{3/2}\sqrt{\ln(d^2) + \ln\left(\frac{1}{\epsilon}\right)},\\
    &T_1^{OU}(\epsilon,\alpha,\beta) := \frac{8\mathfrak{p}_0\mathfrak{l}_{\max}\beta\left((\mathfrak{l}_{\min} - k^2) + \beta\mathfrak{l}_{\max}\right)}{\mathfrak{m}(\mathfrak{l}_{\min}-k^2)^2}\left[\ln\left(\alpha \right)+ \ln\left(\frac{1}{\epsilon}\right)\right],
\end{align*}
where the intuition behind these constants mirrors that of their counterparts in the general linear model. All quantities appearing in these expressions are defined in Section~\ref{sss: assumptionsOU} and in the proofs of Theorems~\ref{TheoremTPP_OU}–\ref{TheoremTP_OU}. In particular, we set $\alpha = 21^{2s}2d(d^{4s} \wedge(ed^2)/(2s))^{2s}$ and $\beta = 9(5 + 4/\gamma)^2$.

Next result provides a control over the probability of the intersection of $\mathcal{T}$, $\mathcal{T}'$ and $\mathcal{T}''$ in terms of the defined quantities for the OU case. 
\begin{theorem}\label{TheoremUnionTausOU}
    Under \cref{AOU} - \cref{LminLmaxCinf}, for a fixed vale of $\epsilon \in \left(0,1/3\right)$,  for all $\lambda > \max\{\lambda_1^{OU}, \lambda_2^{OU}\}$, $T > T_1^{OU}(3\epsilon/4,\alpha,\beta)$ and $0 < k \leq \sqrt{\mathfrak{l}_{\min}}$, it holds that 
    \begin{equation*}
        \P_{A_0}(\mathcal{T} \cap \mathcal{T}' \cap \mathcal{T}'' )\geq 1-3\epsilon.
    \end{equation*}
\end{theorem}
\begin{proof}
    The result comes directly from combining Theorem \ref{TheoremTPP_OU}, Theorem \ref{TheoremT_OU} and  Theorem \ref{TheoremTP_OU}.
\end{proof}
Again, the following error bounds are an immediate consequence of Theorem \ref{TheoremUnionTausOU}. 

\begin{corollary}\label{CoroErrorBoundsOU}
    Assume conditions \cref{AOU}, \cref{LminLmaxCinf} and $\|A_0\|_0=s.$  For a fixed value of $\epsilon \in (0,1/3)$, if $\lambda > \max\{\lambda_1^{OU},\lambda_2^{OU}\}$, $T > T_1^{OU}(3\epsilon/4,\alpha,\beta)$ and  $0 < k <\sqrt{\mathfrak{l}_{\min}}$,   with probability at least $1-3\epsilon$, we have that 
    \begin{equation}\label{l2bound_OU}
         \|\widehat{A}_L -A_0 \|_2^2 \leq \frac{4\lambda^2 s (2 + \gamma)^2}{k^4 \gamma \dn^2}
    \end{equation}
    and
    \begin{equation*}
     \|\widehat{A}_L -A_0 \|_1 \leq  \frac{8\lambda s (1+\gamma)(2+\gamma)}{k^2  \gamma^{\frac{3}{2}}\dn}.
    \end{equation*}
\end{corollary}
\begin{proof}
   We omit the proof for brevity, as the procedure is analogous to that of Corollary \ref{CoroErrorBounds}.
\end{proof}
From Corollary \ref{CoroErrorBoundsOU}, we can specify the rate of convergence of the estimator, highlighting the cases where the discretization of the process dominates. In particular, we state in the next result two distinct scenarios. In one of these, as in the general linear drift case, the lack of continuous observations significantly impacts the rate of the estimator in contrast to the situation where the full path is available.
\begin{corollary}\label{consistencyResult_OU} Assume conditions \cref{AOU} and \cref{LminLmaxCinf}. Let us consider that  $\|A_0\|_0 = s$. For a fixed value of $\epsilon \in (0,1/3)$, if $\lambda > \max\{\lambda_1^{OU},\lambda_2^{OU}\}$, $0 < k <\sqrt{\mathfrak{l}_{\min}}$ and  $T > T_1^{OU}(3\epsilon/4,\alpha,\beta)$, with probability at least $1-3\epsilon$ we have:
\begin{itemize}
    \item[(i)] If $\ln(d^2)/(n\dn ) \rightarrow 0$ and $1/( d^4 n \dn^2 ) \rightarrow 0$, then the error term dominates the convergence of the estimator and 
   \begin{equation}\label{rateconvl2_OU}
        \|\widehat{A}_L -A_0 \|_2 = O_{\P}\left(\frac{d^2}{k^2} \sqrt{s \dn\ln(d^2)}\right).
   \end{equation}
   \item[(ii)] If $\ln(d^2)/(n\dn ) \rightarrow 0$ and $1/( d^4 n \dn^2 ) \rightarrow \infty$, the error term is negligible and 
   \begin{equation}\label{rateconvl22_OU}
        \|\widehat{A}_L -A_0 \|_2 = O_{\P}\left(\frac{1}{k^2}\sqrt{\frac{ s \ln(d^2)}{n \dn}}\right).
   \end{equation}
   \end{itemize}
\end{corollary}
\begin{proof}
    Under aforementioned assumptions, if the parameters of the model satisfy condition $\ln(d^2)/( n\dn) \rightarrow 0$ and $1/(  d^4 n \dn^2) \rightarrow 0$, we have that $\max\{\lambda_1^{OU},\lambda_2^{OU}\} = \lambda_2^{OU}.$ Therefore, from $\eqref{l2bound_OU}$, we obtain $\eqref{rateconvl2_OU}$.
    On the other hand, if $\ln(d^2)/ (n \dn) \rightarrow 0$ and $ 1/ (d^4 n \dn^2) \rightarrow \infty$, the maximum between $\lambda_1^{OU}$ and $\lambda_2^{OU}$ is given by $\lambda_{1}^{OU}$. Therefore, the term associated with the empirical part of the process is the one who dominates and, from $\eqref{l2bound_OU}$, we obtain $\eqref{rateconvl22_OU}$. 
\end{proof}

\begin{remark}
\normalfont
The key distinction between Corollary \ref{consistencyResult} and Corollary \ref{consistencyResult_OU} lies in the way the sparsity parameter $s$ impacts the rates of convergence. This distinction arises primarily due to two factors. First, the concentration inequality differs between the two cases, leading to variations in how the parameters are tracked. Second, in the general linear drift model, the inherent structure of the drift allows us to effectively exploit the sparsity of the vector, making it easier to track the influence of $s$ in the convergence rates, {thus highlighting the advantage of sparsity, particularly in regimes where $s \ll d$}. However, in the Ornstein-Uhlenbeck process, the matrix structure {together with the structural requirement $s \geq d$ imposed by \cref{AOU}} complicates the direct tracing of the sparsity parameter. Nevertheless, when we evaluate the effect of sampling, the impact on the convergence rates is essentially the same in both models, ensuring consistency across our findings.

\end{remark}

\section{The concentration inequalities}{\label{s: CI}}

In order to  control the sets on which our oracle inequality holds true, we will need to derive bounds for the probability of the sets $\mathcal{T}, \mathcal{T}'$ and $\mathcal{T}''$. Given the structure of these sets, it is evident that we will need to employ a concentration bound for functionals of the form
\begin{equation}\label{functionalCI}
    \frac{1}{n} \sum_{i=1}^n\left( F(X_{t_{i}}) - \E\left[F(X_{t_{i}})\right]\right),
\end{equation}
where $F : \R^d \rightarrow \R$ satisfies specific assumptions, depending on the goal and the nature of the problem, and $X_{t_1},\dots,X_{t_n}$ are observations of the stochastic differential equation $\eqref{model}$. In particular, due to the structure of our sets, we will provide a concentration inequality under the assumption that $F$ is Lipschitz.

Concentration inequalities (CI) for additive functionals are crucial probabilistic tools in statistics and other areas as machine learning or mathematical finance. There is increasing focus on concentration inequalities as in general terms, they allow us to quantify the deviation of an estimator or a stochastic element from its expected value or a target quantity. While CI for independent or strong mixed data have been explored for many years (see \cite{ sam00, Zhang20}), deep results for additive functionals of diffusion processes are limited.
Remarkable results on this area are presented in \cite{Str21, djellout2004transportation, Tro23, varvenne2019concentration} among others.  To the best of our knowledge, none of them have explicitly tracked the influence of the discretization step $\dn$ of the observations, nor the dimension of the  process. {As understanding the role of $\Delta_n$ is crucial for assessing the influence of various parameters in our high-dimensional setting, we derive two exponential concentration inequalities corresponding to the two drift structures considered in this work. The first one is established for the general linear drift under \cref{A4}. Since the OU process does not satisfy this assumption, we develop a separate, tailored concentration inequality using tools from Malliavin calculus.} \\

\subsection{Concentration inequality for the general linear drift model}
{We begin by studying the general linear drift setting, where the concentration inequality we derive is inspired by the approach developed in Varvenne’s work \cite{varvenne2019concentration}. In that paper, a similar strategy is applied in the (more involving) context of diffusions driven by fractional Brownian motion. Our contribution lies in simplifying this methodology to the standard Brownian setting, while explicitly and carefully tracking the influence of the discretization step \( \Delta_n \), which plays a crucial role in our analysis. A key assumption for deriving the concentration inequality is the Lipschitz continuity of the contrast functional in \eqref{functionalCI}, which motivates \cref{A4}. Before stating the main result, we introduce the following notation.}

For a fixed value $n \in \N$, we consider  $x \in (\R^d)^n$, such that $x = (x_1,\dots,x_n)$.  Given two elements $x,y \in (\R^d)^n$, we define $d_n$  as \begin{equation*}
    d_n(x,y)^2 = \sum_{i=1}^n \| x_i - y_i \|_2^2.
\end{equation*}
Let  $F:((\R^d)^n,d_n) \rightarrow (\R, |\cdot|)$ be a Lipschitz function, we  define the Lipschitz norm of $F$ by 
\begin{equation*}
\| F\|_{\text{Lip}} = \sup_{x \neq y} \frac{|F(x)-F(y)|}{d_n(x,y)}.
\end{equation*}

\begin{proposition}\label{ConcentrationInequality}  Assume that the drift function satisfies   $\cref{A1}-\cref{A2}$. Then, for all Lipschitz functionals $F:\left((\R^d)^n;d_n\right) \rightarrow \left(\R;|\cdot|\right)$,  $\forall \eta>0$ we have that
\begin{equation*}
     \E\left[\exp{(\eta(F_X - \E[F_X]))}\right] \leq \exp{\left(\frac{16\eta^2 d^2 \| F\|_{\text{\rm Lip}}^2 n \dn \exp(4L\dn) }{(1-\exp(-M \dn))^2}\right)},
\end{equation*}
where $F_X = F(X_{t_1},\dots,X_{t_n})$ and $L,M >0$ are constants from \cref{A1}-\cref{A2}.\\\\
In addition, if  $F = \frac{1}{n}\sum_{i=1}^n f(X_{t_i})$ for a Lipschitz function $f: \R^d \rightarrow \R$, it is
\begin{equation*}\label{CIgoodbound}
     \E\left[\exp\left(\frac{\eta}{n}\sum_{i=1}^n (f(X_{t_i}) -\E[f(X_{t_i})]) \right)\right] \leq \exp{\left(\frac{16\eta^2 d^2 \| f\|_{\text{\rm Lip}}^2  \dn \exp(4L\dn) }{n(1-\exp(-M \dn))^2}\right)},
\end{equation*}
and then,
\begin{equation*}
     \P\left(\frac{1}{n}\sum_{i=1}^n \left(f(X_{t_i}) -\E[f(X_{t_i})]\right) > r\right) \leq \exp{\left(\frac{-r^2 n (1-\exp(-M \dn))^2 }{64 d^2 \| f\|_{\text{\rm Lip}}^2  \dn \exp(4L\dn)}\right)}.
\end{equation*}
\end{proposition}

\noindent
The proof of this Proposition can be consulted in Section \ref{s: proofCI}.

\subsection{Concentration inequality for the OU case}
In contrast to the general linear model, the OU process does not satisfy \cref{A4}. Therefore, a different strategy must be adopted to present a suitable concentration inequality. In this case, our approach is inspired by the work of \cite{CMP20}, where techniques from Malliavin calculus are employed to handle a similar challenge.

We now present some preliminaries on Malliavin calculus. For a comprehensive treatment of the subject, we refer the reader to \cite{nualart2006malliavin}.

Let \(\mathbb{H}\) be a real separable Hilbert space. We denote by \(B := \{ B(h) : h \in \mathbb{H} \}\) an isonormal Gaussian process over \(\mathbb{H}\), meaning that \(B\) is a centered Gaussian family with covariance structure
\[
\E[B(h_1) B(h_2)] = \langle h_1, h_2 \rangle_{\mathbb{H}}, \quad \text{for all } h_1, h_2 \in \mathbb{H}.
\]

Let \(\mathcal{S}\) denote the set of smooth cylindrical random variables of the form \\ \(F = f(B(h_1), \ldots, B(h_n))\), where \(n \geq 1\), \(f: \mathbb{R}^n \rightarrow \mathbb{R}\) is a \(C^{\infty}\)-function with compact support, and \(h_i \in \mathbb{H}\).

The Malliavin derivative \(D F\) of \(F\) is defined as
\[
D F := \sum_{i=1}^n \frac{\partial f}{\partial x_i}\left(B(h_1), \ldots, B(h_n)\right) h_i.
\]
We introduce the norm
\[
\|F\|_{1,2}^2 := \mathbb{E}[F^2] + \mathbb{E}[\|D F\|_{\mathbb{H}}^2],
\]
and define \(\mathbb{D}^{1,2}\) as the closure of \(\mathcal{S}\) with respect to this norm.

A key property of the Malliavin derivative \(D\) is that it satisfies a chain rule. Specifically, if \(\varphi: \mathbb{R}^n \rightarrow \mathbb{R}\) is a continuously differentiable function with bounded partial derivatives (i.e., \(\varphi \in C_b^1\)) and \((F_1, \ldots, F_n)\) is a vector of elements in \(\mathbb{D}^{1,2}\), then \(\varphi(F_1, \ldots, F_n) \in \mathbb{D}^{1,2}\), and
\[
D \varphi(F_1, \ldots, F_n) = \sum_{i=1}^n \frac{\partial \varphi}{\partial x_i}(F_1, \ldots, F_n) D F_i.
\]

The concentration inequality we will use follows from the following key result, which provides tail bounds for certain elements in the space \(\mathbb{D}^{1,2}\).

\begin{theorem}\label{TheoremIvan}(\cite{Nou09}, Theorem 4.1).
Assume that the random variable $Z \in \mathbb{D}^{1,2}$ and define the function
\[
g_Z(z) := \mathbb{E} \left[ \langle D Z, - D L^{-1} Z \rangle_{\mathbb{H}} \mid Z = z \right],
\]
where $L$ denotes the generator of the so-called Ornstein-Uhlenbeck semigroup and $D$ is the Malliavin derivative operator defined in  a Hilbert space denoted by $\mathbb{H}$, as above.
Suppose that the following conditions hold for some $\alpha \geq 0$ and $\beta > 0$:
\begin{enumerate}
    \item[(i)] $g_Z(Z) \leq \alpha Z + \beta$ holds $\P$-almost surely,
    \item[(ii)] The law of $Z$ has a Lebesgue density.
\end{enumerate}
Then, for any $z > 0$, it holds that
\begin{equation*}
\P(Z \geq z) \leq \exp\left( -\frac{z^2}{2\alpha z + 2\beta} \right) \hspace{0.3cm}\text{and} \hspace{0.3cm} \P(Z \leq -z) \leq \exp\left( -\frac{z^2}{2\beta} \right).
\end{equation*}
\end{theorem}
The following concentration inequality is derived by applying Theorem  \ref{TheoremIvan} to a quadratic form of the OU process, related to the empirical covariance matrix.
\begin{proposition}\label{PropCI_OU}
    Under  \cref{AOU} and \cref{LminLmaxCinf}, for all $x> 0$, it holds that 
       \begin{equation}\label{CIOUbruta}
        \sup_{ v \in \R^d, \|v\|_2 =1} \P_{A_0}\left(\left|v^{\star}(C_T - C_{\infty})v \right| > x\right) \leq 2\exp\left(-\frac{x^2 n(1 - \exp(-\mathfrak{m}\dn))(x+\mathfrak{l}_{\max})^{-1}}{8\mathfrak{p}_0\mathfrak{l}_{\max}\left(1 - \exp(-(n+1)\mathfrak{m}\dn)\right)}\right).
    \end{equation}
    In particular, when $\dn \rightarrow 0$,
    \begin{equation*}
        \sup_{ v \in \R^d, \|v\|_2 =1} \P_{A_0}\left(\left|v^{\star}(C_T - C_{\infty})v \right| > x\right) \leq 2\exp(-n\dn H_0(x))
    \end{equation*}
    where the function $H_0$ is defined as 
    \begin{equation*}
        H_0(x) = \frac{\mathfrak{m}}{8\mathfrak{p}_0\mathfrak{l}_{\max}}\frac{x^2}{x + \mathfrak{l}_{\max}}.
    \end{equation*}
\end{proposition}

\section{Numerical results}{\label{s: num}}
In this section, we aim to validate our theoretical results by applying them to a simulated process. Specifically, we compare the estimation performance of the maximum likelihood estimator (MLE) and the proposed Lasso estimator for the drift parameter of a diffusion process in a high-dimensional setting. Additionally, we present a study of the $l_1$ and $l_2$ errors of the estimators as the dimension of the parameter increases.
Using the YUIMA package \cite{yuima}, we will simulate a $d$-dimensional diffusion process that satisfies $\eqref{model}$ and whose drift function is given by 
\begin{equation}\label{diffSim}
    b_{\theta_0}(x) = 3sx + \sum_{j=1}^p \theta_0^j \cos\left((j+1)x\right), \qquad x\in \mathbb{R}^d,
\end{equation}
where $\theta_0 \in \R^p$. In particular, for our simulations, we set the true parameter values $ \theta_0^i $ to be uniformly distributed between $ 2 $ and $ 3 $, ensuring that equation \eqref{diffSim} satisfies the model assumptions listed in Subsection \ref{ss: ass}. Furthermore, we clarify that the cosine function in \eqref{diffSim} is applied componentwise for $ x \in \mathbb{R}^p $.

{The drift function in our simulated process consists of a linear component combined with oscillatory terms, introducing periodic fluctuations in the system's dynamics. These oscillations can model external periodic influences, spatially structured environments, or interactions with an underlying periodic potential. Such formulations appear across various disciplines. In biology, they describe diffusion processes with periodic coefficients, relevant to intracellular transport under structured force fields \cite{Rei02}. Similarly, in quantum mechanics, periodic perturbations in drift terms serve as approximations for semiclassical particle motion in optical or magnetic confinement systems \cite{Fot05}.}

In the first experiment, we aim to compare the performance of the estimators in terms of both estimation accuracy and support recovery. Our goal is to evaluate whether the Lasso estimator captures the sparsity structure more effectively than the classical MLE. We simulate a 10-dimensional diffusion process with a drift of the form \eqref{diffSim}, where true parameter $\theta_0$ belongs to $\mathbb{R}^{30}$, i.e., we set $p = 30$. The time horizon is fixed at $T = 7$, and the true parameter $\theta_0$ is generated with 70\% sparsity, meaning that only 30\% of its components are non-zero, and these values are uniformly distributed between 2 and 3. The process is discretized using a time step of $\Delta_n = 0.01$, as further reductions in the discretization step did not yield significant improvements, a phenomenon also observed in \cite{Gai19}. 
Figure~\ref{fig:trajectories} presents the simulated trajectories of the process across its 10 dimensions. This visualization provides a clearer understanding of the diffusion process, offering insight into its dynamics within the higher-dimensional context. We point out that the selection of  the  tuning parameter $\lambda$ has been done by using the cross-validation criteria described in Section 5 of \cite{Dex22}.
\begin{figure}[H]
\centering
\includegraphics[width=450pt]{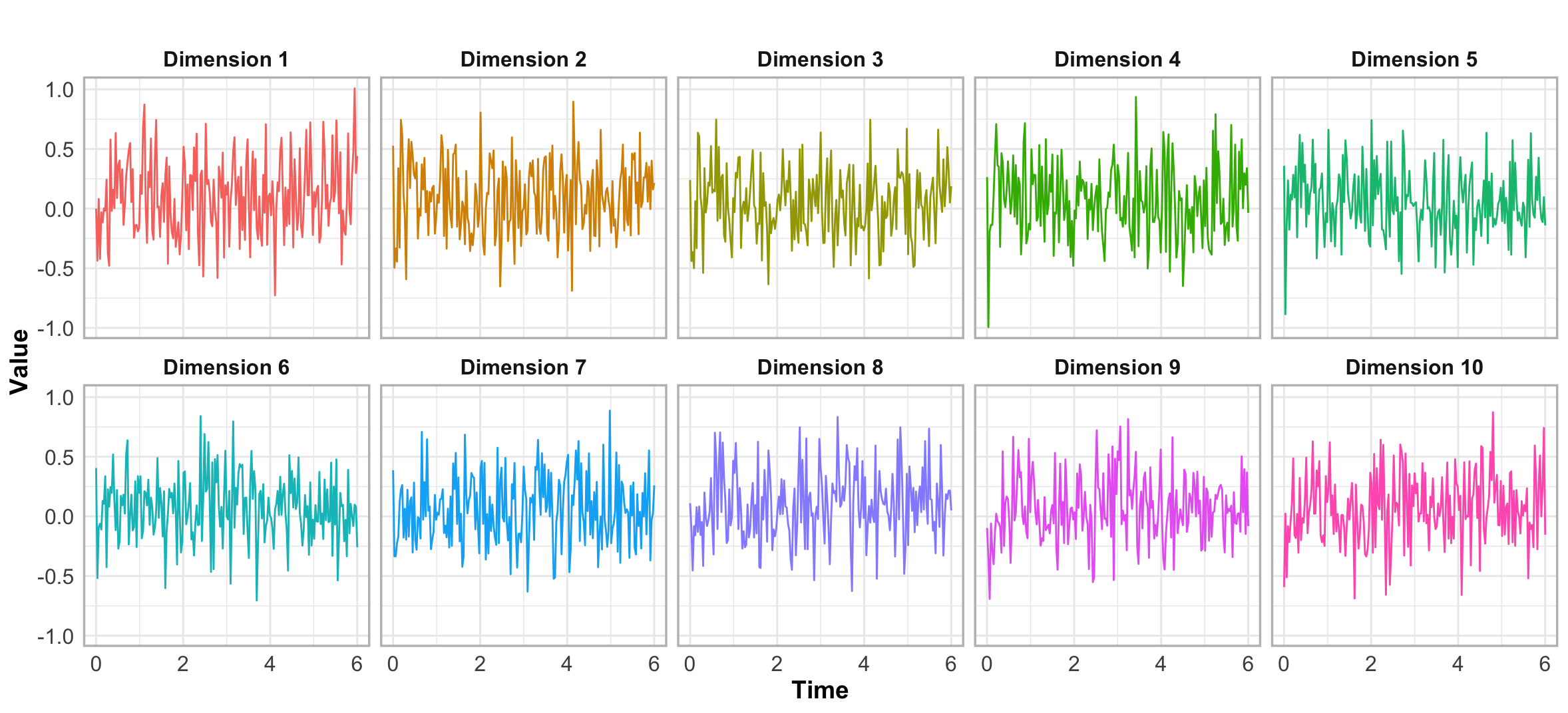}
\caption{Simulated sample paths of the diffusion process, shown across its different dimensions. Each graphic represents the evolution of the process for a concrete dimension.}
    \label{fig:trajectories}
\end{figure}
\noindent
We present our results in Figure~\ref{fig: matrix}. To compare the performance of the MLE and the Lasso estimator,  we have used a heatmap to emphasize the non-zero values of the parameter. We observe that the  estimation obtained through the Lasso regularisation is significantly better than the  classical MLE, particularly in terms of support recovery, even for a relatively small dimension of true parameter.

\begin{figure}[H]
\centering
\includegraphics[width=450pt]{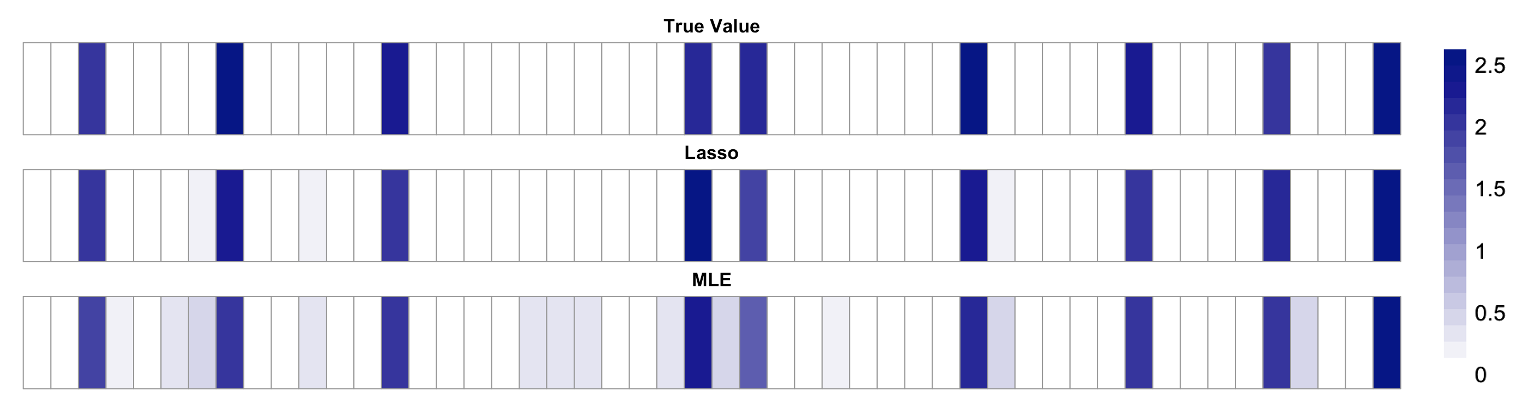}
\caption{Comparison of the sparse true parameter and the estimated parameter using MLE  and Lasso Estimator.}
    \label{fig: matrix}
\end{figure}
\noindent
In the second experiment, we analyse the $l_1$ and $l_2$ estimation errors of the MLE and Lasso as the dimension of the parameter $p$ varies. Our objective is to assess how both estimators scale when increasing the dimensionality of $\theta_0$. 
For this purpose, we simulate a $10$-dimensional diffusion process with a drift function of the form \eqref{diffSim}, for different values of $p$, ranging from $10$ to $50$, while keeping the time horizon fixed at $T = 5$. The true parameter $\theta_0$ is generated with 80\% sparsity, and the non-zero values are drawn uniformly between 2 and 3. For each value of $p$, the error is computed over 30 iterations, where in all the cases, we set a discretization step of $\Delta_n = 0.01$, and the tuning parameter $\lambda$  is selected via cross-validation. Results are presented in Figure~\ref{fig:L1 grafico} and Figure~\ref{fig:L2 grafico}. We observe that in both cases, the Lasso estimator outperforms the classical estimator, thus validating the theoretical results.

\begin{figure}[H]
\centering
    \includegraphics[width=450pt]{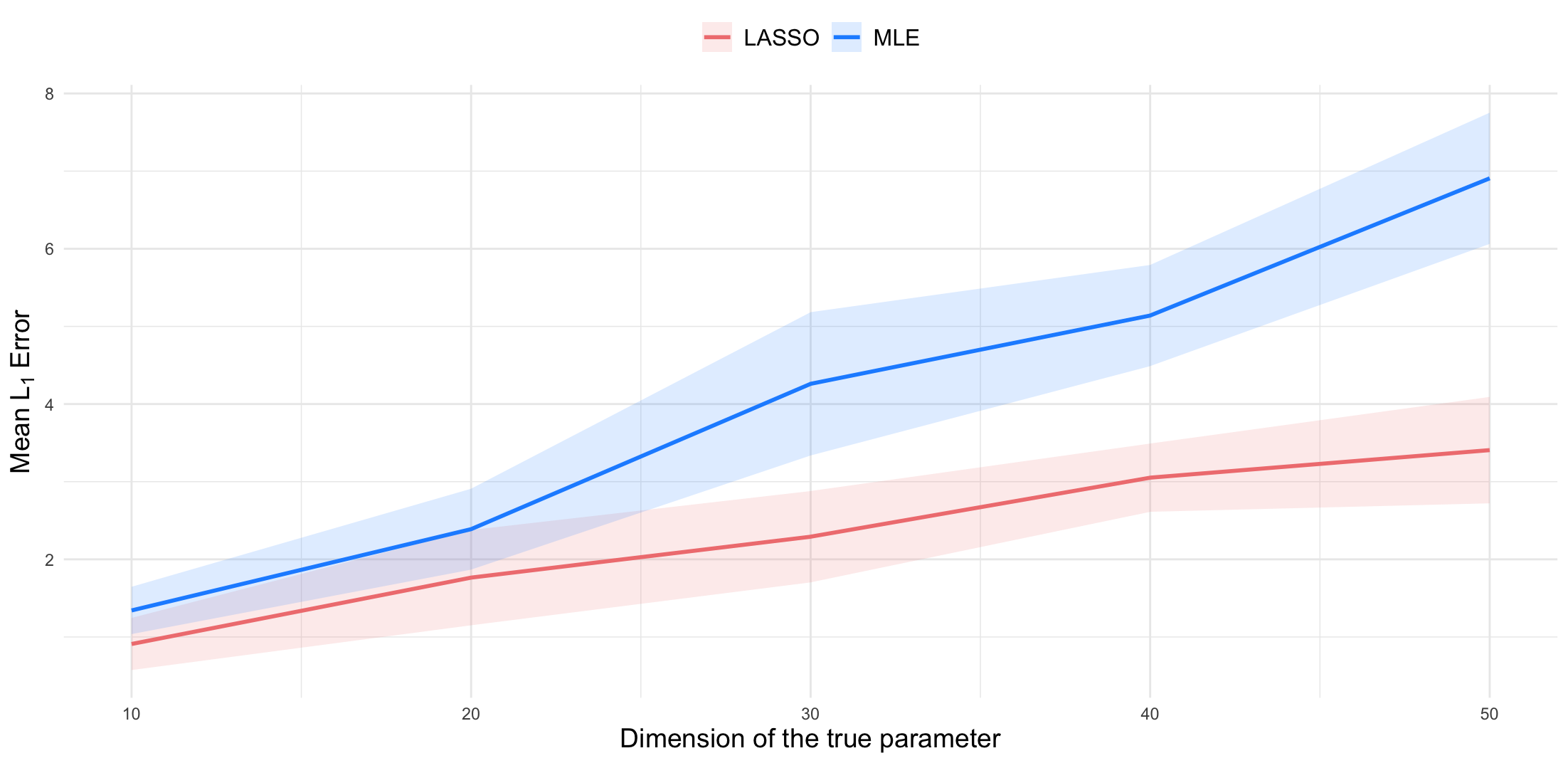}
    \caption{$l_1$ mean error of the MLE and the Lasso estimator $\pm$ one standard deviation.}
    \label{fig:L1 grafico}
\end{figure}
\begin{figure}[H]
\centering
    \includegraphics[width=450pt]{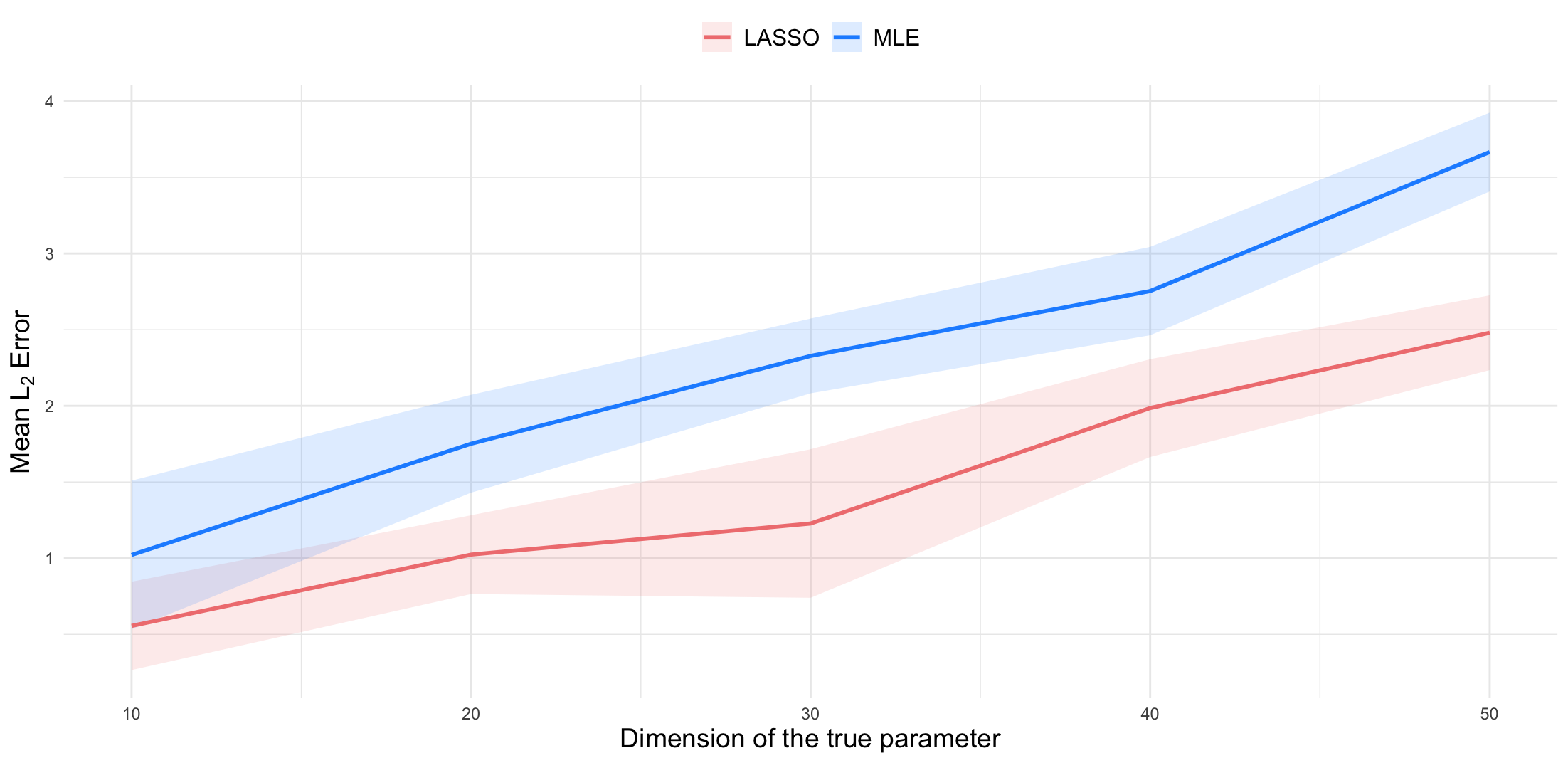}
    \caption{$l_2$ mean error of the MLE and the Lasso estimator $\pm$ one standard deviation.}
    \label{fig:L2 grafico}
\end{figure}

\section{Conclusions and outlook}\label{s: CAO}
Our work addresses the problem of estimating the drift parameters in high-dimensional diffusion processes observed at discrete time points. To this end, we propose the use of the Lasso estimator under a sparsity assumption on the unknown parameter. The theoretical analysis is centered around an oracle inequality for the Lasso estimator, presented in Theorem~\ref{TheoremOracleInequality}. This result holds on the intersection of three specifically constructed sets: one related to the martingale component of the process ($\mathcal{T}$), another capturing the discretization error ($\mathcal{T}'$), and a third ($\mathcal{T}''$)
involving the model design through a compatibility condition. Among these, the set $\mathcal{T}'$, which accounts for the discretization bias, represents a central methodological contribution of our work, as it addresses a gap in the analysis of Lasso estimation for discretely observed diffusion processes.

We establish control over these sets using concentration inequalities developed in Propositions~\ref{ConcentrationInequality} and~\ref{PropCI_OU}. These results explicitly quantify the dependence on the discretization step, the model dimension, the number of parameters, and the sparsity level, thus providing a significant advancement in the theory for discretely observed systems. Under suitable asymptotic conditions, we show that the effect of discretization can become negligible, allowing the Lasso estimator to attain the same optimal convergence rates as in the ideal case of continuous observation. These findings are formally stated in Corollaries~\ref{CoroErrorBounds} and~\ref{CoroErrorBoundsOU}, corresponding to the general linear model and the Ornstein–Uhlenbeck process, respectively.

In addition, numerical experiments presented in Section~\ref{s: num} validate our theoretical results. They demonstrate that the Lasso estimator performs well in both $l_1$ and $l_2$ error norms and significantly outperforms the maximum likelihood estimator in terms of support recovery.

A natural and promising direction for future research is the extension of our framework to nonlinear drift functions. While the oracle inequality~\eqref{OracleIneq} and the structure of the sets $\mathcal{T}$, $\mathcal{T}'$ and $\mathcal{T}''$ 
remain formally valid, the main challenge lies in the probabilistic control of these sets. In contrast to linear models, the derivative of the drift in the nonlinear setting may depend explicitly on the parameter, requiring uniform bounds. This, in turn, calls for more sophisticated tools such as Talagrand’s generic chaining technique. Although such methods have been successfully employed in continuous-time settings (see, e.g. \cite{marklasso}), their adaptation to the discrete-time, high-dimensional regime remains an open and technically demanding problem.

Another interesting avenue for future work is the exploration of alternative matrix structures in high-dimensional OU processes. While our current analysis assumes a full-rank interaction matrix, it would be valuable to investigate the implications of relaxing Assumption~\ref{AOU} to allow for low-rank structures. In such cases, traditional methods may no longer be adequate, and estimators incorporating nuclear norm penalization could provide a more appropriate framework for capturing the underlying low-rank dynamics.


\section{Proof of main results}{\label{s: proof main}}
In this section, we provide the proof of our main results. 

\subsection{Proof of the    oracle inequality}
We will start by giving a detailed proof of Theorem $\ref{TheoremOracleInequality}$.
\begin{proof}\label{ProofOracleInequality}
    We decompose the contrast function $\eqref{contrastFunction}$ as 
    \begin{align*}
    R_T(\theta)=&\frac{1}{T}\sum_{i=1}^{n}\ \| \Delta X_i \|_2^2 
    + \frac{2}{n}\sum_{i=1}^{n} \bt(X_{t_{i-1}})^{\star} \Delta X_i
    + {\dn} \|\bt - \bto\|_D^2  \\
    &-\dn \|\bto\|_D^2 
    + 2\dn \langle \bt, \bto\rangle_D. 
    \end{align*}
    For any $\theta \in \Theta$, we can express the difference of $R_T(\w{\theta}_L) - R_T(\theta)$ as
    \begin{equation*}
      R_T(\w{\theta}_L) - R_T(\theta)=\dn \|\bth - \bto\|_D^2 -   \dn \|\bt - \bto\|_D^2
        + 2\left\langle \bth  - \bt  , \Delta X_i + \dn \bto\right\rangle_D.
    \end{equation*}
    By the  definition of $\w{\theta}_L$,  we have that $R_T(\w{\theta}_L) - R_T(\theta) \leq \lambda(\|\theta\|_1 - \|\w\theta_L\|_1)$. Therefore, for any  $\theta \in \Theta$, we obtain the following inequality:
    \begin{equation}\label{ineq1}
        \dn \|\bth - \bto\|_D^2 \leq   \dn \|\bt - \bto\|_D^2
        + 2\left\langle \bt - \bth  , \Delta X_i + \dn \bto\right\rangle_D 
        + \lambda\left(\|\theta\|_1 - \|\w\theta_L\|_1\right).
    \end{equation}
    Due to the identity $\Delta X_i = -\int_{t_{i-1}}^{t_i} \bto(X_s)ds + \Delta W_i$, the second term of the right side of inequality  $\eqref{ineq1}$ can be seen as
    \begin{align*}
       & \frac{2}{n} \sum_{i=1}^{n} \left(\bt(X_{t_{i-1}}) - \bth(X_{t_{i-1}}) \right)^{\star}  \left( \int_{t_{i-1}}^{t_i} \left(\bto(X_{t_{i-1}}) - \bto(X_s) \right)  ds \right) \\
       & + \frac{2}{n} \sum_{i=1}^{n} \left( \bt(X_{t_{i-1}}) - \bth(X_{t_{i-1}}) \right)^{\star} \Delta W_i .
    \end{align*}
Then, using the random functions that we have introduced in  $\eqref{rf1}$ and $\eqref{rf2}$,  inequality $\eqref{ineq1}$ can be expressed as
    \begin{equation*}
        \dn \|\bth - \bto\|_D^2 \leq   \dn \|\bt - \bto\|_D^2
        +  2A(\theta,\w\theta_L) + 2 G(\theta,\w\theta_L)
        + \lambda\left(\|\theta\|_1 - \|\w\theta_L\|_1\right).
    \end{equation*}
 Observe that the following inequalities hold true
    \begin{align*}
        |G(\theta,\w\theta_L)| 
        &\leq  \| \theta - \w\theta_L \|_1 \left\| \sup_{\theta' \in \Theta} \frac{1}{n} \sum_{i=1}^{n} b_{\theta'}(X_{t_{i-1}})^{\star}\Delta W_i\right\|_{\infty}
    \end{align*}
    and
    \begin{align*}
    |A(\theta,\w\theta_L)| 
    & \leq  \| \theta - \w\theta_L \|_1 \left\|   \sup_{\theta' \in \Theta} \frac{1}{n} \sum_{i=1}^{n}  b_{\theta'}(X_{t_{i-1}})^{\star} \left(\int_{t_{i-1}}^{t_i} \left( \bto(X_s) - \bto(X_{t_{i-1}})\right)ds\right)\right\|_{\infty}.
    \end{align*}
    Therefore, on $\mathcal{T} \cap \mathcal{T}'$, we obtain that 
    \begin{equation*}
         \dn \|\bth - \bto\|_D^2 \leq   \dn \|\bt - \bto\|_D^2
        + \lambda\left(\|\theta\|_1 - \|\w\theta_L\|_1 + \| \theta - \w\theta_L\|_1\right).
    \end{equation*}
    Let $\mathcal{S}(\theta)=\text{supp}(\theta)$. Then, 
    \begin{equation*}
        \|\theta\|_1 - \|\w\theta_L\|_1 + \| \theta - \w\theta_L\|_1 \leq  2\|\w\theta_L|_{\mathcal{S(\theta)}} - \theta\|_1
    \end{equation*}
    and hence
    \begin{equation}\label{ineq2}
           \dn \|\bth - \bto\|_D^2  + \lambda\| \w\theta_L- \theta\|_1\leq   \dn \|\bt - \bto\|_D^2
        + 4\lambda \|\w\theta_L|_{\mathcal{S(\theta)}} - \theta\|_1.
    \end{equation}
    Now, given  a constant $\gamma >0$, we will consider two scenarios.  If we are in the case where  $4\lambda \|\w\theta_L|_{\mathcal{S(\theta)}} - \theta\|_1 \leq \gamma \dn  \|\bt - \bto\|_D^2$, then, from $\eqref{ineq2}$, we obtain that $$\|\bth - \bto\|_D^2  \leq (1+\gamma)  \|\bt - \bto\|_D^2.$$
    If instead we have that  $4\lambda \|\w\theta_L|_{\mathcal{S(\theta)}} - \theta\|_1 > \gamma \dn  \|\bt - \bto\|_D^2$, we observe from $\eqref{ineq2}$ that 
    \begin{equation*}
        \lambda\| \w\theta_L- \theta\|_1\leq \left(4 + \frac{4}{\gamma}\right)\|\w\theta_L|_{\mathcal{S(\theta)}} - \theta\|_1,
    \end{equation*}
    which implies that $\w\theta_L - \theta \in \mathcal{C}(s,3 + 4/\gamma)$. Applying Cauchy--Schwarz inequality in \eqref{ineq2}, we obtain  that 
    \begin{equation}\label{ineqq1}
           \dn \|\bth - \bto\|_D^2  + \lambda\| \w\theta_L- \theta\|_1\leq   \dn \|\bt - \bto\|_D^2
        + 4\lambda \sqrt{s}\|\w\theta_L - \theta\|_2.
    \end{equation}
Now, on $\mathcal{T} \cap \mathcal{T}' \cap \mathcal{T}''$, we have that $\|\bth - \bto\|_D \geq k \| \w\theta_L- \theta\|_2$. Then, we can express  inequality \eqref{ineqq1} as
    \begin{equation}\label{ineqq2}
                \|\bth - \bto\|_D^2  \leq   \|\bt - \bto\|_D^2
        +\frac{ 4\lambda \sqrt{s}}{k \dn }\left(\| \bth - \bto\|_D + \| \bt - \bto\|_D\right).
    \end{equation}
    Using the inequality $2ab \leq \epsilon^{-1}a^2 + \epsilon b^2$ with $a=\lambda \sqrt{s}$,  $b=\| \bth - \bto\|_D$ and $b=\| \bt - \bto\|_D$, we can write \eqref{ineqq2} as
    \begin{equation*}
        \|\bth - \bto\|_D^2 \leq \frac{k\dn +2\epsilon}{k\dn - 2\epsilon}\|\bt - \bto\|_D^2 + \frac{4\lambda^2s}{\epsilon k \dn}\left(1 - \frac{2\epsilon}{k\dn}\right)^{-1}.
    \end{equation*}
    Finally, to match with the case in where  $4\lambda \|\w\theta_L|_{\mathcal{S(\theta)}} - \theta\|_1 \leq \gamma \dn  \|\bt - \bto\|_D^2$, we choose $\epsilon=k\gamma\dn/2(2+\gamma)$, and it leads to 
    \begin{equation*}
        \|\bth - \bto\|_D^2  \leq  (1 + \gamma)\|\bt - \bto\|_D^2 + \frac{4\lambda^2 s (2 + \gamma)^2}{k^2 \gamma \dn^2},
    \end{equation*}
    which completes the proof.
\end{proof}

\subsection{Controlling probabilities of the sets $\mathcal{T}, \mathcal{T}'$ and $\mathcal{T}''$ for the general linear drift}\label{sectionControllingProbabilities}
In this section, we proceed to detail the procedure for gaining control over the probabilities of the  sets $\mathcal{T}, \mathcal{T}'$ and $\mathcal{T}''$. 

\subsubsection{The martingale set $\mathcal{T}$}
We seek to obtain  a  bound for the probability of the set 
\begin{equation*}
     \mathcal{T} = \left\{ \left\|  \frac{1}{n} \sum_{i=1}^{n} \Phi(X_{t_{i-1}})^{\star}\Delta W_i\right\|_{\infty} \leq \frac{\lambda}{4}\right\}.
\end{equation*}
To prove our bound, we  start by defining the following univariate martingale 
\begin{equation}\label{mtgterm1}
    M_n^j = \frac{1}{\sqrt{n}}\sum_{i=1}^n \phi_j(X_{t_{i-1}})^{\star} \Delta W_i.
\end{equation}
To simplify notation, let us denote by $\Tilde{\phi_j} : \R^d \rightarrow \R$ the function  $\Tilde{\phi_j}(x) := \| \phi_j(x)\|_2^2$, and for $j=0,\dots,p, \hspace{0.1cm} i=1,\dots, n$ and $k=1,\dots,d$, we define the following quantities:
\begin{align}
    &H_{\dn} = \max_{j,k} \frac{32 \dn^2  \exp(4L\dn) \| \Tilde{\phi_j}^k \|_{\text{Lip}}^2}{(1-\exp(-M \dn))^2}, \label{quantitiesSetT}\\
     &G = \max_{j}\E_{\theta_0}\left[\|\phi_j\|_D^2\right], \label{G}\\
    &R := \max_{i,j,k}2C_j^2\left(1 + \E_{\theta_0}\left[|X^k_{t_{i-1}}|^2\right]\right) , \label{R}
\end{align}
  We point out that, using $R$, we can establish an upper bound for $G$ that tracks an explicit dependence on the dimension of the diffusion, $d$, since
\begin{align}{\label{eq: control G}}
    G &= \max_{j}\E_{\theta_0}\left[\frac{1}{n} \sum_{i=1}^n\|\phi_j(X_{t_{i-1}})\|_2^2\right] \leq \max_{i,j}\E_{\theta_0}\left[\|\phi_j(X_{t_{i-1}})\|_2^2\right] \\
    &\leq \max_{i,j}\E_{\theta_0}\left[C_j^2(1 + \|X_{t_{i-1}}\|_2)^2\right] \leq dR. \nonumber
\end{align}
It is important to note that, under our assumptions, one can easily verify that \(\max_{j,k} \| \tilde{\phi_j}^k \|_{\text{Lip}}^2\), \(\max_{j} C_j^2\), and \(\max_{i,k} \E_{\theta_0}[|X^k_{t_{i-1}}|^2]\) are uniformly bounded in all quantities of interest, namely \(p\), \(d\), \(n\), and \(T\).
Therefore, from \eqref{eq: control G}, we control $G$ through an explicit dependence on the dimension of the process and a finite quantity that does not depend on  $d,p,T$ and $\dn$.

\begin{remark} 
\normalfont
The quantities defined in $\eqref{quantitiesSetT} - \eqref{R}$ play a role in controlling the probability of set $\mathcal{T}$. However, it is important to point out that as $\dn \rightarrow 0$, $H_{\dn}$ converges to a constant. It implies  that there is no hidden dependence either in the discretization step or in the dimension of the process. 
\end{remark}

\noindent
The following result provides exponential bounds for the probability of  $\eqref{mtgterm1}$.
\begin{proposition}\label{propTmtg}
    Let us consider \cref{A1}-\cref{A4}. Then,  
        \begin{equation}\label{PboundMtg}
    \P_{\theta_0}(|M_n^j| > u) \leq 
     \begin{cases}
         2\exp{\left(-\frac{u^2}{32dR\dn }\right)} &\text{if}\quad 0 < u \leq 8\dn R^{3/2} \sqrt{\frac{n }{H_{\dn}}}, \\ 
        2\exp{\left(-\frac{3}{8d} \sqrt[3]{\frac{n u^4}{ H_{\dn} \dn} }    \right)}
        &\text{if}\quad u>8\dn R^{3/2} \sqrt{\frac{n }{H_{\dn}}}.\\
     \end{cases}
\end{equation}
\end{proposition}

\begin{proof}

The quadratic variation of $M_n^j$, denoted by $\langle M^j\rangle_n$, is given by 
\begin{equation*}
    \langle M^j\rangle_n = \frac{\dn}{n} \sum_{i=1}^n \| \phi_j(X_{t_i}) \|_2^2  =\dn\|\phi_j\|_D^2  .
\end{equation*}
Recall that the  exponential martingale property claims that, for $a>0$,
\begin{equation*}
    \E_{}\left[\exp\left(aM_n^j - \frac{a^2}{2}\langle M^j\rangle_n\right)\right] =1.
\end{equation*}
Using it together with an application of Cauchy--Schwarz inequality, we obtain that 
\begin{align*}
    \E_{\theta_0}\left[\exp(\mu M_n^j)\right] &\leq \E_{\theta_0}\left[\exp{(\mu M_n^j - \mu^2\langle M^j\rangle_n)} \exp(\mu^2\langle M^j\rangle_n))\right]\\
    &\leq \E_{\theta_0}\left[\exp{(2(\mu M_n^j - \mu^2\langle M^j\rangle_n))}\right]^{\frac{1}{2}} \E_{\theta_0}\left[\exp(2\mu^2\langle M^j\rangle_n))\right]^{\frac{1}{2}}\\
    &= \E_{\theta_0}\left[\exp(2\mu^2\langle M^j\rangle_n))\right]^{\frac{1}{2}} = \E_{\theta_0}\left[\exp(2\mu^2 \dn\|\phi_j\|_D^2 )\right]^{\frac{1}{2}}.
\end{align*}
A straightforward calculation shows that 
\begin{align}\label{ineq24}
     \E_{\theta_0}\left[\exp(\mu M_n^j)\right]^2 & \leq \E_{\theta_0}\left[\exp{\left( 2 \mu^2 \dn (\|\phi_j\|_D^2 -\E_{\theta_0}\left[\|\phi_j\|_D^2\right])\right)}\exp{\left( 2\mu^2 \dn \E_{\theta_0}\left[\|\phi_j\|_D^2\right]\right)}\right] \nonumber \\
     &= \E_{\theta_0}\left[\exp{\left( 2 \mu^2 \dn (\|\phi_j\|_D^2 -\E_{\theta_0}\left[\|\phi_j\|_D^2\right])\right)}\right]\exp{\left( 2\mu^2 \dn \E_{\theta_0}\left[\|\phi_j\|_D^2\right]\right)} \nonumber \\
     &= \E_{\theta_0}\left[\exp{\left( \frac{2 \mu^2 \dn}{n}\sum_{i=1}^n(\|\phi_j(X_{t_{i-1}})\|_2^2 -\E_{\theta_0}\left[\|\phi_j(X_{t_{i-1}})\|_2^2\right])\right)}\right]\nonumber \\ &\times \exp{\left( 2\mu^2 \dn \E_{\theta_0}\left[\|\phi_j(X_0)\|_2^2\right]\right)}.
     \end{align}
Applying the concentration inequality obtained in Proposition $\ref{ConcentrationInequality}$ to bound $\eqref{ineq24}$, and using the quantities defined in \eqref{quantitiesSetT}, we get

\begin{equation}\label{expBoundMartingale}
    \E_{\theta_0}\left[ \exp(\mu M_n^j)\right]  \leq \exp\left(\frac{\dn d^3 H_{\dn}}{n}\mu^4 + \dn dR\mu^2\right).
\end{equation}
Now,  we are going to obtain a bound for $\P_{\theta_0}(M_n^j > x)$.
Using Markov's inequality and $\eqref{expBoundMartingale}$, for all $\mu >0$, we obtain that
\begin{equation}\label{ineq5}
    \P_{\theta_0}(M_n^j > u) = \P_{\theta_0}(\exp \left(\mu M_n^j\right) > \exp \left(\mu u\right))\leq \exp\left(\frac{\dn d^3 H_{\dn}}{n}\mu^4 + \dn d R \mu^2 - \mu u\right).
\end{equation}
If $0< \mu \leq \sqrt{n R/(d^2 H_{\dn})}$, the minimum of the right side of the inequality $\eqref{ineq5}$ is attained for $\mu^{\star} = u/(4 \dn d R)$.  Therefore,
\begin{equation}\label{b1}
      \P_{\theta_0}(M_n^j > u) \leq  \exp\left(2\dn d R \mu^2 - \mu u\right) \leq \exp{\left(-\frac{u^2}{8dR\dn }\right)} \hspace{0.1cm} 
\end{equation}
 if $0< u \leq 4\dn R^{3/2} \sqrt{n/H_{\dn}}.$
Following the same procedure, in case $ \mu \geq \sqrt{n R/(d^2 H_{\dn})}$, the minimum of the right side of $\eqref{ineq5}$ is attained for $\mu^{\star} =\left(n u/(8 d^3 H_{\dn} \dn)\right)^{1/3}$, and consequently,
\begin{equation}\label{b2}
\P_{\theta_0}(M_n^j > u) \leq \exp\left(\frac{2 \dn d^3 H_{\dn}}{n}\mu^4  - \mu u\right) \leq  \exp{\left(-\frac{3}{8} \sqrt[3]{\frac{ n u^4}{d^3 H_{\dn} \dn} }    \right)}  
\end{equation}
if $ u > 8\dn R^{3/2} \sqrt{{n }/{H_{\dn}}}$. 
Finally, for the case when  $u$ is between $ 4\dn R^{3/2} \sqrt{n /H_{\dn}}$ and $8\dn R^{3/2} \sqrt{n /H_{\dn}}$, we rescale inequality $\eqref{b1}$  by $u = u/2$. It leads to  
    \begin{equation}
    \P_{\theta_0}(|M_n^j| > u) \leq 
     \begin{cases}
         2\exp{\left(-\frac{u^2}{32dR\dn }\right)} &\text{if}\quad 0 < u \leq 8\dn R^{3/2} \sqrt{\frac{n }{H_{\dn}}}, \\ 
        2\exp{\left(-\frac{3}{8} \sqrt[3]{\frac{n u^4}{d^3 H_{\dn} \dn} }    \right)}
        &\text{if}\quad u>8\dn R^{3/2} \sqrt{\frac{n }{H_{\dn}}},\\
     \end{cases}
\end{equation}
that proves our desired result.
\end{proof}

\noindent
Now, we will use the concentration inequality in $\eqref{PboundMtg}$ to control the probability of $\mathcal{T}$.
\begin{theorem}\label{TheoremT} Under \cref{A1}-\cref{A4}, for a fixed value of $\epsilon \in (0,1)$ and for all $\lambda > \lambda_1$, it holds that 
    \begin{equation*}
        \P_{\theta_0}(\mathcal{T}) \geq 1-\epsilon.
    \end{equation*}
\end{theorem}
\begin{proof}
    
The proof is based on the result of Proposition $\ref{propTmtg}$. We have that
\begin{align*}
\P_{\theta_0}\left(\mathcal{T}^C\right) &= \P_{\theta_0}\left(  \left\|  \frac{1}{n} \sum_{i=1}^{n} \Phi(X_{t_{i-1}})^{\star}\Delta W_i\right\|_{\infty} > \frac{\lambda}{4} \right)  \\
    & \leq p \max_{j}\P_{\theta_0} \left(  \left|   \frac{1}{n} \sum_{i=1}^{n} \phi_j(X_{t_{i-1}})^{\star}\Delta W_i \right| > \frac{\lambda}{4} \right)\\
    & \leq  p \max_{j} \P_{\theta_0} \left( |M_n^j| \geq \frac{\sqrt{n}\lambda}{4}\right).
\end{align*}
Now, if $\sqrt{n}\lambda/4 \leq 8\dn R^{3/2} \sqrt{n /H_{\dn}}, $ using inequality $\eqref{PboundMtg}$, we obtain that 
\begin{equation*}
\P_{\theta_0}\left(\mathcal{T}^C\right) \leq 2p\exp{\left(-\frac{n \lambda^2}{512 d R \dn}\right)} .
\end{equation*}
Instead, if $\sqrt{n}\lambda/4 > 8\dn R^{3/2} \sqrt{n /H_{\dn}},$
we have that 
\begin{equation*}
\P_{\theta_0}\left(\mathcal{T}^C\right) \leq 2p\exp{\left(-\frac{3}{32d}\left(\frac{n^3 \lambda^4}{4 H_{\dn} \dn}\right)^{\frac{1}{3}}\right)}.
\end{equation*}
Therefore, if we define 
\begin{equation*}
    \lambda_1 = \max \left\{ 23 \sqrt{\frac{d R \dn}{n} \left[\ln{(2p)} +\ln{\left( \frac{2}{\epsilon}\right)} \right]}, 7 \sqrt[4]{\frac{d^3 H_{\dn} \dn}{n^3} \left[\ln{(2p)} +\ln{\left(\frac{2}{\epsilon}\right)} \right]^{3}}\right\},
\end{equation*}
for a fix value of $\epsilon \in (0,1)$, and for all  $\lambda \geq \lambda_1$, it implies that 
\begin{equation*}
      \P_{\theta_0}(\mathcal{T}) \geq 1-\epsilon,
\end{equation*}
obtaining the desired result.
\end{proof}
\subsubsection{The approximation set $\mathcal{T}'$}
In this subsection, we present the study to control the probability of the set $\mathcal{T}'$, where we consider the error term that comes from  the discretization. As the drift function is of the form $\eqref{linearModel},$ the set $\mathcal{T}'$ defined in $\eqref{settaup}$ can be written as
\begin{equation*}
     \mathcal{T}' = \left\{ \left\| \frac{1}{n} \sum_{i=1}^{n} \Phi(X_{t_{i-1}})^{\star} \left( \sum_{k=0}^p \theta_0^k\int_{t_{i-1}}^{t_i} \left(\phi_k(X_s) - \phi_k(X_{t_{i-1}})\right)ds\right)\right\|_{\infty} \leq \frac{\lambda}{4}\right\},
\end{equation*}
where $\theta_0^k$ denotes the $k$th component of the parameter $\theta_0$, setting $\theta_0^0 = 1$.
In order to control the probability of $\mathcal{T}'$, for each $j \in \{1,\dots,p\}$ we define the functional $F_j : \left(\R^d\right)^n \rightarrow \R$ as 
\begin{equation}\label{functionalF}
    F_j = \frac{1}{\sqrt{n}} \sum_{i=1}^{n} \phi_{j}(X_{t_{i-1}})^{\star} \left( \sum_{k=0}^p \theta_0^k\int_{t_{i-1}}^{t_i} \left(\phi_k(X_s) - \phi_k(X_{t_{i-1}})\right)ds\right).
\end{equation}
Our first result provides a bound for the tail probability of $\eqref{functionalF}$.
\begin{proposition}\label{propTapprox}
    Let us consider \cref{A1}-\cref{A4}. Then, there exists a constant $C_b>0$ independent from \(p\), \(d\), \(n\) and \(T\) such that
 \begin{equation}\label{boundPt2}
    \P_{\theta_0}\left(|F_j| \geq u \right) \leq \exp\left(-\frac{u^2}{4\exp(2)C_b {n}s^2 d^2 \dn^{3}}\right).
\end{equation}
\end{proposition}  
\begin{proof}

 To simplify the calculations, for each $t_i$, we can express functional $\eqref{functionalF}$ as $F_j = \frac{1}{\sqrt{n}}\sum_{i=1}^n F_j^i$, where 
\begin{equation*}
    F_j^i =  \phi_j(X_{t_{i-1}})^{\star} \left( \sum_{k=0}^p \theta_0^k\int_{t_{i-1}}^{t_i} \left(\phi_k(X_s) - \phi_k(X_{t_{i-1}})\right)ds\right).
\end{equation*}
We seek to find an upper bound for $\E_{\theta_0}[|F_j^i|^r]$, for $r\geq 1$. We observe that, for any  $i \in \{1,\dots,n\}$,  Cauchy--Schwarz inequality implies, 
\begin{align*}
\E_{\theta_0}\left[\left|F_j^i\right|^r\right] &= \E_{\theta_0}\left[\left|  \phi_j(X_{t_{i-1}})^{\star}\left( \sum_{k=0}^p \theta_0^k\int_{t_{i-1}}^{t_i} \left(\phi_k(X_s) - \phi_k(X_{t_{i-1}})\right)ds \right)\right|^r\right] \\
    &\leq \E_{\theta_0}\left[ \left\| \phi_j(X_{t_{i-1}})\right\|_2^{2r} \right]^{\frac{1}{2}} \E_{\theta_0}\left[ \left\| \sum_{k=0}^p \theta_0^k \int_{t_{i-1}}^{t_i} \left(\phi_k(X_s) - \phi_k(X_{t_{i-1}})\right)ds \right\|_2^{2r} \right]^{\frac{1}{2}}.
\end{align*}
Since the $\phi_j$'s have linear growth, for all $r \geq 1$, $j=0,\dots,p; i=1,\dots,n; k=1,\dots,d$, we can bound $\E_{\theta_0}[ \left\| \phi_j(X_{t_{i-1}})\right\|_2^{2r}]$ following a similar reasoning as in \eqref{R}. More precisely, if we define 
\begin{equation*}
    J_1(r) = 2^{r-1} \max_{i,j,k} \sup_{s \in [t_{i-1}, t_i]}C_j^{r}\left(1 + \E_{\theta_0}\left[|X_{s}^k|^{r}\right]\right),
\end{equation*}
we obtain that $\E_{\theta_0}[ \| \phi_j(X_{t_{i-1}})\|_2^{2r} ]^{1/2}  \leq d^{{r}/{2}}J_1(2r)^{1/2}$. Using \cref{A1} and applying Jensen's inequality,  for all $j\in \{1,\dots,p\}$ we have that 
\begin{align*}
\E_{\theta_0}\left[\left|F_j^i\right|^r\right] &\leq  d^{r/2}J_1(2r)^{1/2}\E_{\theta_0}\left[ \left\| s \max_{k} \theta_0^k\int_{t_{i-1}}^{t_i} \left(\phi_k(X_s) - \phi_k(X_{t_{i-1}})\right)ds \right\|_2^{2r} \right]^{\frac{1}{2}}\\
    &\leq   d^{r/2} J_1(2r)^{1/2} s^r \dn^{\frac{2r-1}{2}} \max_{k} \left|\theta_0^k\right|^r \E_{\theta_0}\left[  \int_{t_{i-1}}^{t_i} \left\|\phi_k(X_s) - \phi_k(X_{t_{i-1}})\right\|_2^{2r}ds  \right]^{\frac{1}{2}}\\
    & \leq  d^{r/2} J_1(2r)^{1/2} s^r \dn^{\frac{2r-1}{2}} \max_{k} \left|\theta_0^k L_k\right|^r
    \left(\int_{t_{i-1}}^{t_i} \E_{\theta_0}\left[ \left\|X_s - X_{t_{i-1}}\right\|_2^{2r} \right] ds\right)^{\frac{1}{2}}.
\end{align*}
Now we proceed to bound  $\E_{\theta_0}[ \left\|X_s - X_{t_{i-1}}\right\|_2^{2r} ].$ 
By the dynamics of the process and using the inequality $|a+b|^p \leq 2^{p-1}(|a|^p + |b|^p)$ for $p\geq 1$, we obtain that 
\begin{align}\label{ineq6}
    \E_{\theta_0}\left[ \left\|X_s - X_{t_{i-1}}\right\|_2^{2r} \right] &= \E_{\theta_0}\left[ \left\| -\int_{t_{i-1}}^{s}  \bto(X_u)du + \int_{t_{i-1}}^{s} dW_u\right\|_2^{2r}\right] \nonumber \\
    &\leq 2^{2r-1} \E_{\theta_0}\left[ \left\| \int_{t_{i-1}}^{s} \bto(X_u)du\right\|_2^{2r} \right]  + 2^{2r-1} \E_{\theta_0}\left[ \left\|\int_{t_{i-1}}^{s} dW_u\right\|_2^{2r}  \right].
\end{align}
Again, thanks to the linear growth of  the $\phi_j$'s functions, the first term of the right side of the inequality \eqref{ineq6} can be bounded by
\begin{align}\label{bound1}
    \E_{\theta_0}\left[ \left\| \int_{t_{i-1}}^{s}  \bto(X_u)du\right\|_2^{2r} \right] &\leq \dn^{2r-1} \int_{t_{i-1}}^{s} \E_{\theta_0}\left[\left\|   \sum_{k=0}^p\theta_0^k \phi_k(X_u)\right\|_2^{2r} \right] du  \nonumber \\ 
    &\leq \dn^{2r-1}s^{2r}  \max_{k }\left|\theta_0^k\right|^{2r} \int_{t_{i-1}}^{s} \E_{\theta_0}\left[\left\|   \phi_k(X_u)\right\|_2^{2r} \right] du \nonumber \\ 
    &\leq \dn^{2r} d^{r} s^{2r} \max_{k }\left|\theta_0^k\right|^{2r} J_1(2r).
\end{align}
For the second term of the right hand of $\eqref{ineq6}$, using that  $\E[|X|] = \int_{0}^\infty \P(|X| \geq x) dx$ and applying the change of variable $x = z^{2r}$, we obtain that 
\begin{align*}
    \E_{\theta_0}\left[ \left\|\int_{t_{i-1}}^{s} dW_u\right\|_2^{2r} \right] &\leq 2r d^r \int_{0}^\infty \P(|Z| \geq z) z^{2r-1}dz \\
    &\leq 4rd^r \int_{t_{i-1}}^{s} \exp\left(-\frac{z^2}{2(s-t_{i-1})}\right)z^{2r-1}dz,
\end{align*}
where $Z$ is a centred Gaussian 1-dimensional random variable with variance $s- t_{i-1}$. Finally, if we do the change of variable $\tilde{z} = z^2/(2(s-t_{i-1}))$, we obtain that 
\begin{equation}\label{bound2}
    \E_{\theta_0}\left[ \left\|\int_{t_{i-1}}^{s} dW_u\right\|_2^{2r} \right]  \leq  d^r  \dn^r J_2(r), 
\end{equation}
where 
\begin{equation}\label{J2}
J_2(r) := 2^{r+1} \Gamma(r+1)
\end{equation}
for all $r\geq 1$, and $\Gamma(r) = \int_0^\infty \exp(-s)s^{r-1}ds$ denotes the gamma function. Using the fact that for $\dn \rightarrow 0$, the term $\eqref{bound1}$ is negligible with respect to the term $\eqref{bound2}$, we have that
\begin{align*}  
\E_{\theta_0}\left[\left|F_j^i\right|^r\right] &\leq 2^r d^{r/2} J_1(2r)^{1/2} s^r \dn^{\frac{2r-1}{2}} \max_{k} \left|\theta_0^k L_k\right|^r \left(\int_{t_{i-1}}^{s} d^r \dn^r J_2(r) ds\right)^{\frac{1}{2}} \\
    & \leq 2^{r}  d^{r} s^r \dn^{\frac{3r}{2}} \max_{k} \left|\theta_0^k L_k\right|^r (J_1(2r)J_2(r))^{1/2}
     \\
    & \leq \left(2ds \dn^{3/2}\right)^r J(r)^{1/2},
\end{align*}
where $J(r) = \max_{k} \left|\theta_0^k L_k\right|^{2r} J_1(2r)J_2(r).$
To obtain a bound for $\E[|F_j|^r]$, using Jensen inequality, it leads to
\begin{equation}\label{boundEofFjLinearCase}
    \E_{\theta_0}\left[|F_j|^r\right] \leq \E_{\theta_0}\left[\left|\frac{1}{\sqrt{n}}\sum_{i=1}^n F_j^i\right|^r\right] \leq 
    \frac{1}{\sqrt{n^r}} n^{r-1}\sum_{i=1}^n \E_{\theta_0}\left[\left| F_j^i \right|^r\right]  \leq (2\sqrt{n}sd \dn^{3/2})^r J(r)^{1/2}.
\end{equation}
Due to the fact that $\theta_0$, $L_k$ and the $r$-th moments of the ergodic process are uniformly bounded with respect to the parameters of the model for all $r \geq 1$ (see Remark \ref{remarkBoundedUnif}), we obtain that $J(r) \leq C^rJ_2(r)$. Moreover, by Stirling approximation, we can establish a control over the behaviour of the gamma function for all $r \geq 1$, to conclude that 
\begin{equation}\label{boundFj}
    \E_{\theta_0}\left[|F_j|^r\right] \leq (2\sqrt{n}sd \dn^{3/2})^r C_b^{r/2}r^{r/2},
\end{equation}
for a constant $C_b$ independent from $p,d,n$ and $T$. Finally, as the $r$-th moments of $F_j$ are controlled by \eqref{boundFj}, using Markov's inequality with $f(x) = x^r$, we conclude that
\begin{equation*}
    \P_{\theta_0}\left(|F_j| \geq 2eC_b^{1/2}\sqrt{n}s d \dn^{3/2}r\right) \leq \exp(-r^2).
\end{equation*}
From here we deduce the desired result.
\end{proof}

\noindent
The next theorem, provides a bound for the probability of $\mathcal{T}'.$

\begin{theorem}\label{TheoremTP}
    Under \cref{A1}-\cref{A4},  for a fixed value of $\epsilon \in (0,1)$ and for all $\lambda > \lambda_2,$ it holds that 
    \begin{equation*}
        \P_{\theta_0}(\mathcal{T}') \geq 1-\epsilon.
    \end{equation*}
\end{theorem}
\begin{proof}
\begin{align*}
    \P_{\theta_0}\left(\mathcal{T}'^C\right) &= \P_{\theta_0}\left(  \left\|  \frac{1}{n} \sum_{i=1}^{n} \Phi(X_{t_{i-1}})^{\star} \left( \sum_{k=0}^p \theta_0^k\int_{t_{i-1}}^{t_i} \left(\phi_k(X_s) - \phi_k(X_{t_{i-1}})\right)ds\right)\right\|_{\infty} \geq \frac{\lambda}{4} \right)  \\
    & \leq p \max_{j}\P_{\theta_0} \left(   \frac{1}{n}\left| \sum_{i=1}^{n} \phi_j(X_{t_{i-1}})^{\star} \left( \sum_{k=0}^p \theta_0^k\int_{t_{i-1}}^{t_i} \left(\phi_k(X_s) - \phi_k(X_{t_{i-1}})\right)ds\right)\right| \geq \frac{\lambda}{4} \right)\\
    &\leq  p \max_{j} \P_{\theta_0} \left(|F_j| \geq \frac{\sqrt{n}\lambda}{4}\right).
\end{align*}
Now, using Proposition \ref{propTapprox}, we obtain that
\begin{equation*}
\P_{\theta_0}\left(\mathcal{T}'^C\right) \leq p\exp\left(- \frac{\lambda^2}{64\exp(2)C_b s^2 d^2 \dn^3}\right).
\end{equation*}
Therefore, if we define 
\begin{equation*}
    \lambda_2 = 8eC_b^{1/2}sd\dn^{3/2}\sqrt{\ln(p) + \ln\left(\frac{1}{\epsilon}\right)},
\end{equation*}
for a fix value of $\epsilon \in (0,1)$, and for all  $\lambda \geq \lambda_2$, it implies that 
\begin{equation*}
      \P_{\theta_0}(\mathcal{T}') \geq 1-\epsilon,
\end{equation*}
obtaining the desired result.
\end{proof}

\subsubsection{The $\mathcal{T}''$ set}

First, for $u \in \R^p$ such that $\|u\|_2 =1 $, we define the following random function 
\begin{equation}\label{randomFunction}
    H_u(\theta) = u^{\star} \left(\frac{1}{n}\sum_{i=1}^n \dot{\bt}(X_{t_{i-1}})^{\star}\dot{\bt}(X_{t_{i-1}}) - \E[\dot{\bt}(X_{t_{i-1}})^{\star}\dot{\bt}(X_{t_{i-1}})]\right) u.
\end{equation}
In particular, under the assumption that the drift function is of the form $\eqref{linearModel}$, we can express $\eqref{randomFunction}$ as 
\begin{equation}\label{randomFunctionLinearCase}
    H_u =  u^{\star} \left(\frac{1}{n}\sum_{i=1}^n \Phi(X_{t_{i-1}})^{\star}\Phi(X_{t_{i-1}}) - \E[\Phi(X_{t_{i-1}})^{\star}\Phi(X_{t_{i-1}})]\right) u.
\end{equation}
By \cref{A4}, the functions $\left(\Phi(\cdot)^{\star}\Phi{}(\cdot)\right)_{i,j} = \langle \phi_i,\phi_j \rangle:\R^d \rightarrow \R$ are  Lipschitz continuous. Therefore,   if we denote by $\mathcal{M}_{\text{Lip}}$ the matrix such that $\left(\mathcal{M}_{\text{Lip}}\right)_{i,j} = L_{ij}$, we see that
$\left\|  H_u\right\|_{\text{Lip}} \leq \| \mathcal{M}_{\text{Lip}}\|_{\text{op}}$.
We point out that, due to the linearity in the drift function, \eqref{randomFunctionLinearCase} does not depend on  $\theta$. 
As previously done in the martingale set, we introduce a quantity derived from the concentration inequality:
\begin{equation}\label{K}
    K_{\dn} := \frac{16 \dn^2 \exp(4L\dn )\| \mathcal{M}_{\text{Lip}}\|_{\text{op}}^2}{\left(1-\exp(-M \dn)\right)^2},
\end{equation}
that plays a role in controlling the probability of $\mathcal{T}''$. 
\begin{remark} 
\normalfont
As in Proposition \ref{propTmtg}, we have defined the constant $K_{\Delta_n}$ to control the probability of the compatibility condition set. As before, when $\Delta_n \rightarrow 0$, $K_{\Delta_n}$ converges to a constant dependent only on the model. Therefore, in the limit, the value is independent of $\Delta_n$, $d$, and $p$. 
\end{remark}

\begin{proposition}\label{propTPP}
    Let us consider \cref{A1}-\cref{A4}. Then, for all $z>0$,   
\begin{equation}\label{ConcentrationBoundtaupp}
    \P_{\theta_0}\left(  \frac{\left| H_u \right|}{\| u \|_2} \geq z\right) \leq  2\exp\left(-\frac{z^2 n \dn}{4d^2 K_{\dn} }\right).
\end{equation}
\end{proposition}  
\begin{proof}
By using Markov's inequality, it leads to 
\begin{align}\label{b}
     \P_{\theta_0}\left(  \frac{ H_u }{\| u \|_2} \geq z\right) &= \P_{\theta_0}\left(  \exp\left(\frac{\mu H_u}{\| u \|_2} \right) \geq \exp(\mu z)\right)\nonumber \\
     &\leq \E_{\theta_0}\left[\exp{\left( \frac{\mu H_u}{\| u \|_2}\right)}\right]\exp{(-\mu z)}. 
\end{align}
If we apply the concentration inequality gathered in Proposition  $\ref{ConcentrationInequality}$ to the expectation in \eqref{b}, it implies  that for all $\mu>0$,
\begin{equation}
    \P_{\theta_0}\left(  \frac{\left| H_u \right|}{\| u \|_2} \geq z\right) \leq 2\exp\left(\mu^2\frac{d^2 K_{\dn} }{n \dn} - \mu z\right).
\end{equation}
Finally, choosing  $\mu = \frac{n\dn z}{2 d^2 K_{\dn}}$, we conclude that, for all $z>0$, we have that 
\begin{equation}
    \P_{\theta_0}\left(  \frac{\left| H_u(\theta) \right|}{\| u \|_2} \geq z\right) \leq  2\exp\left(-\frac{z^2 n \dn}{4 d^2  K_{\dn} }\right),
\end{equation}
obtaining the desired result.
\end{proof}

\begin{theorem}\label{TheoremTPP}
    Under \cref{A1}-\cref{A5}, for a fixed value of $\epsilon \in (0,1)$,  for all $T > T_1$ and $0 < k < \sqrt{l}$, it holds that 
    \begin{equation*}
        \P_{\theta_0}(\mathcal{T}'') \geq 1-\epsilon.
    \end{equation*}
\end{theorem}
\begin{proof}
To prove our theorem, we start by noticing that 
\begin{align*}
    \frac{\|\bt - b_{\upsilon} \|_{D}^2}{\| \theta - \upsilon\|_2^2} &= \frac{\E\left[ \|\bt - b_{\upsilon} \|_{D}^2\right]}{\| \theta - \upsilon\|_2^2} - \frac{ \E\left[ \|\bt - b_{\upsilon} \|_{D}^2\right] - \|\bt - b_{\upsilon} \|_{D}^2}{\| \theta - \upsilon\|_2^2}.
\end{align*}
 From \cref{A5}  and applying the mean value theorem to  $\E\left[ \|\bt - b_{\upsilon} \|_{D}^2\right] - \|\bt - b_{\upsilon} \|_{D}^2$, we obtain that
\begin{align*}
   \frac{\|\bt - b_{\upsilon} \|_{D}^2}{\| \theta - \upsilon\|_2^2} \geq l +  \frac{ H_{\theta - \upsilon} }{\| \theta - \upsilon\|_2^2},
\end{align*}
where $H_{\theta - \upsilon}$ has been defined in $\eqref{randomFunctionLinearCase}$.
Now, using  properties of the supremum and the infimum, we have that
\begin{align*}\label{ineq7}
    &\P_{\theta_0}\left( \inf_{\substack{\theta \in \R^p  : \| \theta \|_0 =s \\ \upsilon \in \R^p : \theta - \upsilon \in \mathcal{C}(s,3 + 4/\gamma)}}      \frac{\|\bt - b_{\upsilon} \|_{D}^2}{\| \theta - \upsilon\|_2^2}  \geq k^2 \right) \geq  \\
     &\P_{\theta_0}\left( \inf_{\substack{\theta \in \R^p  : \| \theta \|_0 =s \\ \upsilon \in \R^p : \theta - \upsilon \in \mathcal{C}(s,3 + 4/\gamma)}} \frac{ H_{\theta - \upsilon} }{\| \theta - \upsilon\|_2^2}   \geq -(l-k^2) \right) = \\
     &\P_{\theta_0}\left( \sup_{\substack{\theta \in \R^p  : \| \theta \|_0 =s \\ \upsilon \in \R^p : \theta - \upsilon \in \mathcal{C}(s,3 + 4/\gamma)}} \frac{ -H_{\theta - \upsilon} }{\| \theta - \upsilon\|_2^2}   < l-k^2 \right) \geq \nonumber\\
    &1- \P_{\theta_0}\left( \sup_{\substack{\theta \in \R^p  : \| \theta \|_0 =s \\ \upsilon \in \R^p : \theta - \upsilon \in \mathcal{C}(s,3 + 4/\gamma)}} \frac{ |H_{\theta - \upsilon}| }{\| \theta - \upsilon\|_2^2}   \geq l-k^2 \right).
\end{align*}
Therefore, we can conclude that 
\begin{equation}\label{ineqq}
    \P_{\theta_0}\left( \inf_{\substack{\theta \in \R^p  : \| \theta \|_0 =s \\ \upsilon \in \R^p : \theta - \upsilon \in \mathcal{C}(s,3 + 4/\gamma)}}      \frac{\|\bt - b_{\upsilon} \|_{D}^2}{\| \theta - \upsilon\|_2^2}  \geq k^2 \right) \geq 1- \P_{\theta_0}\left( \sup_{ u \in \mathcal{C}(s,3 + 4/\gamma)} \frac{ |H_{u}| }{\| u\|_2^2}   \geq l - k^2 \right).
\end{equation}
Now, to prove our statement, we derive a bound for the probability of the right part of inequality $\eqref{ineqq}$.
For any $p \geq 1, q \geq 0, r > 0$, we define $\mathbb{B}_q(r) := \{ u \in \R^p \hspace{0.1cm}| \hspace{0.1cm} \|u\|_q \leq r\}$ the ball centred at 0 of radius $r$ in norm $l_q$ and the spare set $\mathcal{K}(s) := \mathbb{B}_0(s) \cup \mathbb{B}_2(1).$ For any set $A$, we denote by $\text{conv}\{a\}$ and $cl\{A\}$ the convex hull and the closure of $A$, respectively. 
Working similarly to Lemma F.1 and F.3 from the supplementary material of \cite{Bas15}, we can show that
\begin{align*}
    \P_{\theta_0}\left( \sup_{ u \in \mathcal{C}(s,3 + 4/\gamma)} \frac{ |H_{u}| }{\| u\|_2^2}   \geq l-k^2 \right) &\leq \P_{\theta_0}\left( \sup_{ u \in \left(5 + \frac{4}{\gamma} \right)cl\{\text{conv}(\mathcal{K}(s))\}} \frac{ |H_{u}| }{\| u\|_2^2}   \geq l-k^2 \right) \nonumber\\
    & \leq  \P_{\theta_0}\left( \left(5 + \frac{4}{\gamma} \right)^2 \sup_{ u \in cl\{\text{conv}(\mathcal{K}(s))\}} \frac{ |H_{u}| }{\| u\|_2^2}   \geq l-k^2 \right) \nonumber \\
    &\leq \P_{\theta_0}\left(  \sup_{ u \in \mathcal{K}(2s)} \frac{ |H_{u}| }{\| u\|_2^2}   \geq \frac{l-k^2}{3\left(5 + \frac{4}{\gamma} \right)^2} \right) .
\end{align*}
Following the proof of Lemma F.2 from the supplementary material of \cite{Bas15}, we obtain that 
\begin{equation}\label{ep}
    \P_{\theta_0}\left( \sup_{ u \in \mathcal{C}(s,3 + 4/\gamma)} \frac{ |H_{u}| }{\| u\|_2^2}   \geq l-k^2 \right) \leq   21^{2s}\left(p^{2s} \wedge \left(\frac{ep}{2s}\right)^{2s}\right)\sup_{u \in \Theta}\P_{\theta_0}\left(  \frac{ |H_{u}| }{\| u\|_2^2}   \geq \frac{l-k^2}{9\left(5 + \frac{4}{\gamma}\right)^2} \right).
\end{equation}
Finally, combining $\eqref{ineqq}$ and $\eqref{ep}$, and using Proposition $\ref{propTPP}$, we have that 
\begin{align*}
 \P_{\theta_0}\left( \mathcal{T}''\right)&=\P_{\theta_0}\left( \inf_{\substack{\theta \in \R^p  : \| \theta \|_0 =s \nonumber \\ \upsilon \in \R^p : \theta - \upsilon \in \mathcal{C}(s,3 + 4/\gamma)}}      \frac{\|\bt - b_{\upsilon} \|_{D}^2}{\| \theta - \upsilon\|_2^2}  \geq k^2 \right) \nonumber \\
 &\geq 1- 21^{2s}\left(p^{2s} \wedge \left(\frac{ep}{2s}\right)^{2s}\right)\sup_{u \in \Theta}\P_{\theta_0}\left(  \frac{ |H_{u}| }{\| u\|_2^2}   \geq \frac{l-k^2}{9\left(5 + \frac{4}{\gamma}\right)^2} \right) \nonumber \\
  &\geq 1 - 21^{2s}\left(p^{2s} \wedge \left(\frac{ep}{2s}\right)^{2s}\right)2\exp{\left(\frac{- n  \dn(l-k^2)^2 }{324  \left(5 + \frac{4}{\gamma}\right)^4 d^2 K_{\dn}}\right)}.
\end{align*}
If we define the constant 
\begin{equation*}
    T_1 := \frac{324 d^2 K_{\dn} \left(5 + \frac{4}{\gamma}\right)^4}{(l-k^2)^2}\left[\ln{\left(\frac{2}{\epsilon}\right)} + \ln{\left(21^{2s}\left(p^{2s} \wedge \left(\frac{ep}{2s}\right)^{2s}\right)\right)}\right],
\end{equation*}
for a fixed value of $\epsilon \in (0,1)$, if $n\dn = T \geq T_1$, then 
\begin{equation*}
    \P_{\theta_0}(\mathcal{T}'') \geq 1 - \epsilon,
\end{equation*} 
obtaining the desired result.
\end{proof}

\subsection{Controlling probabilities of the sets $\mathcal{T}, \mathcal{T}'$ and $\mathcal{T}''$ for the OU model}\label{sectionControllingProbabilities_OU}
In this section, we present and prove the results adapted to the OU process. Although there are key differences between the OU model and the linear drift case, the general structure and procedure of the proofs in this section are similar to those detailed in the previous one. For this reason, we present the results in a more direct way. In contrast to the general linear case, we begin by establishing control over the set $\mathcal{T}''$ since it provides valuable input for the analysis of $\mathcal{T}$.

\subsubsection{The $\mathcal{T}''$ set}
Using the bi-linear random form defined in \eqref{bilinearFormOU}, the compatibility condition set $\mathcal{T}''$ can be written for the OU case model \eqref{OUdefinition} as 
\begin{gather*}
    \mathcal{T}'' := \left\{ \inf_{\substack{A \in \M{d}{d}  : \| A \|_0 =s \\ B \in \M{d}{d} : A - B \in \mathcal{C}(s,3 + 4/\gamma)}}{\frac{\|A - B\|_D}{\|A-B\|_2} \geq k}\right\} \\
    =  
    \left\{ \inf_{\substack{ U \in \mathcal{C}(s,3 + 4/\gamma)}}{\frac{\|U\|_D}{\|U\|_2} \geq k}\right\}.
\end{gather*}

\begin{theorem}\label{TheoremTPP_OU}
    Under \cref{AOU} and \cref{LminLmaxCinf}, for a fixed value of $\epsilon \in (0,1)$,  for all $T > T_1^{OU}(\epsilon,\alpha,\beta)$ and $0 < k < \sqrt{\mathfrak{l}_{\min}}$, it holds that 
    \begin{equation*}
        \P_{A_0}(\mathcal{T}'') \geq 1-\epsilon.
    \end{equation*}
\end{theorem}

\begin{proof}
First, we observe that for a matrix $U \in \M{d}{d}$,
\begin{equation*}
    \|U\|_{D}^2 =\frac{1}{n}\sum_{i=1}^n(UX_{t_{i-1}})^{\star}(UX_{t_{i-1}}) = \frac{1}{n}\sum_{i=1}^n\tr(UX_{t_{i-1}}X_{t_{i-1}}^{\star}U^{\star}) = \tr(UC_TU^{\star}).
\end{equation*}
In particular, 
\begin{equation*}
    \frac{\|U\|_{D}^2}{\|U\|_2^2} = \frac{\tr(U(C_{\infty})U^{\star})}{\|U\|_2^2} - \frac{\tr(U(C_{\infty}-C_{T})U^{\star})}{\|U\|_2^2} \geq \mathfrak{l}_{\min} - \frac{\tr(U(C_{\infty}-C_{T})U^{\star})}{\|U\|_2^2},
\end{equation*}
where we can see the analogy with the function $H_u$ defined in \eqref{randomFunctionLinearCase}, in this case considering  $H_U = \tr(U(C_{\infty}-C_{T})U^{\star})$. Therefore, by analogy, following the proof of Theorem \ref{TheoremTPP} (crf  \eqref{ep}), we obtain that 
\begin{align*}
 \P_{A_0}\left( \mathcal{T}''\right)& \geq 1- 21^{2s}\left(d^{4s} \wedge \left(\frac{ed^2}{2s}\right)^{2s}\right)\sup_{U \in \mathbb{A}}\P_{A_0}\left(  \frac{ |H_{U}| }{\| U\|_2^2}   \geq \frac{\mathfrak{l}_{\min}-k^2}{9\left(5 + \frac{4}{\gamma}\right)^2} \right).
\end{align*}
For a fix $U \in \mathbb{A}$, we denote its $j$-th row vector by $u^j$ and $\mathbf{u} = \text{vec}(U) \in \R^{d^2}$. Moreover, if we define a symmetric random matrix $\mathbf{D}_{C_T} = id \otimes (C_\infty - C_ T) \in \M{d^2}{d^2}$, we note that 
\begin{equation*}
\frac{|\tr(U(C_{\infty})U^{\star})|}{\|U\|_2^2} = \frac{|\mathbf{u}^{\star}\mathbf{D}_{C_T}\mathbf{u}|}{\|\mathbf{u}\|_2^2} = \frac{|\sum_{j=1}^{d} u_j (C_\infty - C_T) u_j^\star|}{\sum_{j=1}^{d} \|u_j\|_2^2}.
\end{equation*}
Therefore, using the concentration inequality of Proposition \ref{PropCI_OU}, we obtain that 
\begin{equation}\label{referef}
    \P_{A_0}\left( \mathcal{T}''\right) \geq 1- 21^{2s}\left(d^{4s} \wedge \left(\frac{ed^2}{2s}\right)^{2s}\right)2d\exp\left(-n\dn H_0\left(\frac{\mathfrak{l}_{\min}-k^2}{9(5 +\frac{4}{\gamma})^2}\right)\right).
\end{equation}
Finally, considering $T > T_1^{OU}(\epsilon,\alpha,\beta)$, 
with $\beta=9(5 + 4/\gamma)^2$ and $\alpha = 21^{2s}2d(d^{4s} \wedge (ed^2)/(2s))^{2s}$, we obtain the desired result.
\end{proof}

\subsubsection{The martingale set $\mathcal{T}$}
Observe that, in the OU case, set $\mathcal{T}$ is given explicitly by 
\begin{equation*}
    \mathcal{T} := \left\{ \left\| \frac{1}{n} \sum_{i=1}^{n} \Delta W_i X_{t_{i-1}}^{\star} \right\|_{\infty} \leq \frac{\lambda}{4}\right\}.
\end{equation*}
\begin{theorem}\label{TheoremT_OU} 
Under \cref{AOU} and \cref{LminLmaxCinf}, for a fixed value of $\epsilon \in (0,1)$ and for all $\lambda > \lambda_1^{OU}$ and $T > T_1^{OU}(\epsilon/2,\alpha,\beta)$, it holds that 
    \begin{equation*}
        \P_{A_0}(\mathcal{T}) \geq 1-\epsilon.
    \end{equation*}
\end{theorem}
\begin{proof}
For each $ 1\leq j,k \leq d$, we can define the univariate martingale 
$M_n^{jk} = \frac{1}{\sqrt{n}} \sum_{i=1}^n X_{t_{i-1}}^j\Delta W_{i}^k$, whose quadratic variation satisfies that $\langle M^{jk} \rangle_n = \dn  C_T^{jj}$. Using the union bound property and defining $\mathfrak{a}=\|\text{diag}(C_\infty)\|_{\infty}$, we obtain that
\begin{gather}
    \P_{A_0}\left(\left\|\frac{1}{n} \sum_{i=1}^{n} \Delta W_i X_{t_{i-1}}^{\star} \right\|_{\infty} \geq \frac{\lambda}{4}\right) \leq  \nonumber\\\P_{A_0}\left(\left\|\frac{1}{n} \sum_{i=1}^{n} \Delta W_i X_{t_{i-1}}^{\star} \right\|_{\infty} \geq \frac{\lambda}{4}, \max_{1\leq j,k\leq d}\langle M^{jk} \rangle_n  \leq \dn(\mathfrak{a} + \mathfrak{l}_{\min} - k^2) \right) \nonumber \\ 
    +  \P_{A_0}\left( \max_{1\leq j,k\leq d}\langle M^{jk} \rangle_n  \geq \dn(\mathfrak{a}+ \mathfrak{l}_{\min} - k^2) \right) \leq \nonumber \\
     d^2 \max_{1\leq j,k \leq d}\P_{A_0}\left(|M_n^{jk}| \geq \frac{\lambda\sqrt{n}}{4} , \langle M^{jk} \rangle_n  \leq \dn(\mathfrak{a} + \mathfrak{l}_{\min} - k^2) \right) \label{etiqueta1}\\
    + \P_{A_0}\left(\max_{1\leq j \leq k}|C_T^{jj}| \geq (\mathfrak{a}+\mathfrak{l}_{\min} - k^2) \right) \label{etiqueta2}.
\end{gather}
Since the event 
\begin{gather*}
    \left\{\max_{1\leq j\leq d}|C_T^{jj}| \geq (\mathfrak{a} + \mathfrak{l}_{\min} - k^2)\right\} \subseteq \left\{ \|\text{diag}(C_{\infty} - C_T)\|_{\infty} \geq \mathfrak{l}_{\min} - k^2\right\} \\
    \subseteq \left\{\sup_{U \in C(s,3+4/\gamma) } \frac{|\tr(U(C_{\infty}-C_{T})U^{\star})|}{\|U\|_2^2}\right\},
\end{gather*}
if $T \geq T_1^{OU}(\epsilon/2,\alpha,\beta)$, from Proposition \ref{PropCI_OU}, as bounded in \eqref{referef}, the probability \eqref{etiqueta2} can be bounded by $\epsilon/2$. For controlling term \eqref{etiqueta1}, Bernstein's inequality for discrete martingales, leads to
\begin{equation*}
    \P_{A_0}\left(|M_n^{jk}| \geq \frac{\lambda\sqrt{n}}{4} , \langle M^{jk} \rangle_n  \leq \dn(\mathfrak{a} + \mathfrak{l}_{\min} - k^2) \right) \leq \exp\left(\frac{-\lambda^2n}{32(\mathfrak{a} + \mathfrak{l}_{\min}-k^2)}\right).
\end{equation*}
Therefore, choosing $\lambda \geq \lambda_1^{OU}$, we obtain that
\begin{equation*}
    \P_{A_0}(\mathcal{T}) \geq 1- \P_{A_0}\left(\left\|\frac{1}{n} \sum_{i=1}^{n} \Delta W_i X_{t_{i-1}}^{\star} \right\|_{\infty} \geq \frac{\lambda}{4}\right) \geq 1- \epsilon, 
\end{equation*}
which completes the proof.
\end{proof}

\subsubsection{The approximation set $\mathcal{T}'$}

In the case of the OU process, the approximation set is explicitly expressed as
\begin{equation}\label{matrixAppSet}
    \mathcal{T}' := \left\{ \left\|  \frac{1}{n} \sum_{i=1}^{n} \left( A_0\int_{t_{i-1}}^{t_i}\left( X_s -X_{t_{i-1}}\right)ds\right) X_{t_{i-1}}^*\right\|_{\infty} \leq \frac{\lambda}{4}\right\}.
\end{equation}
Following a similar procedure as for the general linear drift case, for $1\leq j,k \leq d$, if we denote by $a_{jk}$ the element of the $j$-row and $k$-column of $A_0$, {we observe that each entry of the matrix inside the supremum norm in \eqref{matrixAppSet} can be expressed as 
\begin{equation}
    M_{jk} = \frac{1}{n}\sum_{i=1}^nM_{jk}^{i} = \frac{1}{n}\sum_{i=1}^n \sum_{l=1}^d a_{jl}\left(\int_{t_{i-1}}^s (X_s - X_{t_{i-1}})^l ds\right)X_{t_{i-1}}^k.
\end{equation}
Therefore, if we define the univariate process 
\begin{equation*}
    F_{jk} = \frac{1}{\sqrt{n}}\sum_{i=1}^n F_{jk}^i   = \frac{{d}}{\sqrt{n}}\sum_{i=1}^n \max_{jk} |a_{jk}|\left(\int_{t_{i-1}}^s (X_s - X_{t_{i-1}})^j ds\right)X_{t_{i-1}}^k,
\end{equation*}
{we have that all entries of the matrix in \eqref{matrixAppSet} are controlled in absolute value by $|F_{jk}|$.}
\begin{proposition}\label{propTapproxOU}
    Let us consider \cref{AOU} and \cref{LminLmaxCinf}. Then, there exists a constant $C_b^{OU}>0$ independent of $d$, $n$ and $T$ such that
 \begin{equation}\label{boundPt2OU}
    \P_{A_0}\left(|F_{jk}| \geq u \right) \leq \exp\left(-\frac{u^2}{4\exp(2)C_b^{OU} {n} d^4 \dn^{3}}\right).
\end{equation}
\end{proposition}  
\begin{proof}
We proceed to bound $\E[|F_{jk}^i|^r]$ for $r\geq 1$. By analogy with the proof of Proposition \ref{propTapprox}, we consider the quantities 
\begin{equation*}
    J_1^{OU}(r) := \max_{i,k} \sup_{s \in [t_{i-1},t_i]}\E_{A_0}\left[\left|X_s^k\right|\right], 
\end{equation*} 
and $J_2(r)$ as introduced in \eqref{J2}.   Set   $J^{OU}(r) := \max_{j,k}|a_{jk}|^rJ_1^{OU}(2r)J_2(r)$. As shown in \eqref{boundEofFjLinearCase} for the general linear case,  in an analogous manner, it is easy to check that for all $r \geq 1$, we have 
\begin{equation*}
    \E_{A_0}[|F_{jk}|^r]  = \E_{A_0}\left[\left|\frac{1}{\sqrt{n}}\sum_{i=1}^n F_{jk}^i\right|^r\right]\leq \left(2\sqrt{n}d^2\dn^{3/2}\right)^rJ^{OU}(r)^{1/2}.
\end{equation*}
Again,  since all moments of the process are bounded and using Stirling approximation for controlling the Gamma function, it leads to
\begin{equation}
    \E_{A_0}\left[|F_{jk}|^r\right] \leq (2\sqrt{n}d^2 \dn^{3/2})^r (C_b^{OU})^{r/2}r^{r/2},
\end{equation}
for a constant $C_b^{OU}$ independent of $d,n$ and $T$. Finally, using Markov's inequality with $f(x) = x^r$, we obtain the desired result.
\end{proof}

\begin{theorem}\label{TheoremTP_OU}
    Under \cref{AOU} and \cref{LminLmaxCinf},  for a fixed value of $\epsilon \in (0,1)$ and for all $\lambda > \lambda_2^{OU},$ it holds that 
    \begin{equation*}
        \P_{A_0}(\mathcal{T}') \geq 1-\epsilon.
    \end{equation*}
\end{theorem}
\begin{proof}Using the union bound property,
    \begin{equation*}
        \P_{A_0}(\mathcal{T}'^{C}) =  \P_{A_0}\left(\max_{1 \leq j,k \leq d}|F_{jk}| \geq \frac{\sqrt{n}\lambda}{4}\right) \leq d^2 \max_{1 \leq j,k \leq d}\P_{A_0}\left(|F_{jk}| \geq \frac{\sqrt{n}\lambda}{4}\right)
    \end{equation*}
Therefore, using Proposition \ref{propTapproxOU} and choosing $\lambda \geq \lambda_2^{OU},$  the theorem holds.
\end{proof}

\subsection{Proof of the concentration inequalities}{\label{s: proofCI}}
In this section, we will provide the proof of Proposition \ref{ConcentrationInequality} and Proposition \ref{PropCI_OU}.
\subsubsection{Proof of the concentration inequality for the general linear drift}
We start by introducing the following notation.
Given two functions $f,g \in C(\R_+, \R^d)$, let $a,b,c > 0$ such that $a < b< c$. Then, we define 
\begin{equation*}
    f_{[a,b]} \sqcup g_{[b,c]} := 
     \begin{cases}
       f(t) \quad &\text{if}\quad  a\leq t \leq b\\
       g(t) \quad&\text{if} \quad b < t \leq c.\\
     \end{cases}
\end{equation*}
Following the ideas in \cite{varvenne2019concentration}, the proof is based on the decomposition of the functional $F_X = F(X_{t_1},\dots, X_{t_n})$ into a sum of martingale increments, and derive conditional exponential moments for them.
\begin{proof}
    Let $W$ be a $d-$dimensional Brownian motion defined on $\left(\Omega, \mathcal{F}, \P\right).$ Let us denote by $\left\{\mathcal{F}_t\right\}_{t \geq 0}$ the natural filtration associated to $W$ and by $\P_W$ its distribution.
    
    For a  time $t_k > 0$, we denote by $\Tilde{w}_s$ the Brownian motion for $t_s > t_k$ and by $\P_W(d\Tilde{w})$ the distribution law of $\Tilde{w}$ conditioned to  $\mathcal{F}_{t_k}$.\\
    Given $F_X$, we start by constructing the associated martingale $M_k = \E[F_X | \mathcal{F}_{t_k}]$ such that \\$\sum_{k=1}^n \left(M_k - M_{k-1} \right)= F_X - \E[F_X]$.
    Under \cref{A1}, $X_t$ can be expressed as a measurable functional of the time $t$, of the initial condition $X_0$ and of the Brownian motion, i.e. $X_t := \Psi: \R^+ \times \R^d \times \mathcal{C}(\R^+; \R^d) \rightarrow \R^d $ such that $X_t =\Psi(t,x,(W_s)_{s \in [0,t]}) = \Psi_t(x,(W_s)_{s \in [0,t]})$. To streamline the notation in this proof and alleviate the complexity of the upcoming computations, we have opted to simplify by using \( |\cdot| \) to denote both the Euclidean norm on \(\mathbb{R}^d\) and the absolute value on \(\mathbb{R}\). \\ 
    For $k \geq 1$, we have that 
    {
    \begin{align}\label{initialCIequation}
        &|M_k - M_{k-1}| = |\E[F_X| \mathcal{F}_{t_k}] - |\E[F_X| \mathcal{F}_{t_{k-1}}]|\leq  \nonumber\\[1.5 ex]
        & \int_{\Omega} \big| F\big( \Psi_{t_1}(X_0, W_{[0,t_1]}),\dots,\Psi_{t_k}(X_0, W_{[0,t_k]}) ,\Psi_{t_{k+1}}(X_0, W_{[0,t_k]} \sqcup \tilde{w}_{[t_k,t_{k+1}]}),\dots, \nonumber\\[1.5 ex]
        &\Psi_{t_{n}}(X_0, W_{[0,t_k]} \sqcup \tilde{w}_{[t_k,t_n]})\big) \nonumber\\[1.5 ex]
        &- F\big(\Psi_{t_1}(X_0, W_{[0,t_1]}),\dots,\Psi_{t_{k-1}}(X_0, W_{[0,t_{k-1}]}),
        \Psi_{t_{k}}(X_0, W_{[0,t_{k-1}]} \sqcup \tilde{w}_{[t_{k-1},t_{k}]}),\dots, \nonumber\\[1.5 ex] &\Psi_{t_{n}}(X_0, W_{[0,t_{k-1}]} \sqcup \tilde{w}_{[t_{k-1},t_{n}]})\big)\big|\P_W(d\Tilde{w}) \nonumber \\[1.5 ex]
        &  \leq  \|F\|_{\text{Lip}} \int_{\Omega} \sum_{j=k}^n \Big| \Psi_{t_{j}}\left(X_0, W_{[0,t_{k}]} \sqcup \tilde{w}_{[t_{k},t_{j}]}\right) - \Psi_{t_{j}}\left(X_0, W_{[0,t_{k-1}]} \sqcup \tilde{w}_{[t_{k-1},t_{j}]}\right) \Big| \P_W(d\Tilde{w}).
    \end{align}
    }%
   Let us set $u:= j-k+1.$ Then, for $u \in \{1,\dots,n\}$, we define
   \begin{equation*}
      I_{t_u} := \Psi_{t_{u+k-1}}\left(X_0, (W_{s})_{s\in [0,t_{k}]}\sqcup  (\tilde{w}_{s})_{s\in [t_{k},t_{u+k-1}]}\right) 
\end{equation*}
and
\begin{equation*}
    \overline{I_{t_u}} := \Psi_{t_{u+k-1}}\left(X_0, (W_{s})_{s\in [0,t_{k-1}]}\sqcup  (\tilde{w}_{s})_{s\in [t_{k-1},t_{u+k-1}]}\right),
\end{equation*}
where $t_u = u\dn$.\\
Using the notation we have introduced, we can express $I_{t_u}, \overline{I_{t_u}}$ as
    \begin{equation*}
        I_{t_u} = X_0 - \int_{0}^{t_u} b(I_s)ds + \int_0^{t_k}dW_s +  \int_{t_k}^{t_{u+k-1}}d\tilde{w}_s,
    \end{equation*}
    \begin{equation*}
        \overline{I_{t_u}} = X_0 - \int_{0}^{t_u} b(\overline{I_s})ds +  \int_0^{t_{k-1} }dW_s +  \int_{t_{k-1}}^{t_{u+k-1}}d\tilde{w}_s .
    \end{equation*}
    Using the  change of variable $\tilde{s} =s - t_{k-1}$ and defining  the Brownian motions independent of $\mathcal{F}_{t_{k-1}}$, $(W^{(k)}_{s})_{s \geq 0}:= \left(W_{s+k-1} - W_{k-1}\right)_{s \geq 0}$ and  $\tilde{w}^{(k)}:= (\tilde{w}_{s+k-1} - \tilde{w}_{k-1})_{s \geq 0};$  we can express the difference of $I_{t_u}$ and $\overline{I_{t_u}}$  as
    \begin{equation*}
        I_{t_u} - \overline{I_{t_u}} = \int_{0}^{t_u} \left(b(\overline{I_s}) - b(I_s)\right) ds +  \int_{0}^{\dn }d(W^{(k)} - \tilde{w}^{(k)})_{\tilde{s}}.
    \end{equation*}
    We will proceed to control the differences $|I_{t_u} - \overline{I_{t_u}}|$. To do that, we  consider two different cases:  $u \geq 2$ and $1 \leq u \leq 2$.
    
    When $u\geq 2$, by \cref{A2}, we have that
    \begin{equation}\label{gronwall1}
        \frac{d}{dt} |I_{t_u} - \overline{I_{t_u}}|^2 = 2 \left|I_{t_u} - \overline{I_{t_u}}\right|^{\star}\left(b(\overline{I_{t_u}}) - b(I_{t_u})\right) \leq -2M|I_{t_u} - \overline{I_{t_u}}|^2.
    \end{equation}
    Applying Gronwall inequality to $\eqref{gronwall1}$, we obtain that, $\forall t_u \in [t_2,\infty),$
    \begin{equation}\label{ineq3}
        |I_{t_u} - \overline{I_{t_u}}| \leq |I_{t_2} - \overline{I_{t_2}}|\exp{\left(-M(t_u-t_2)\right)}.
    \end{equation}
    We consider now the case $1 \leq u\leq 2$. The Lipschitzianity of $b$ gathered in \cref{A1} ensures 
    \begin{align*}
        |I_{t_u} - \overline{I_{t_u}}| &= \left| \int_0^{t_u} \left(b(\overline{I_{s}}) - b(I_{s}) \right)ds +  \int_0^{\dn} d\left( W^{(k)} - \tilde{w}^{(k)}\right)_{\Tilde{s}}\right| \\
        & \leq L \int_0^{t_u} \left|I_{s} - \overline{I_{s}} \right|ds +  \left| \int_0^{\dn} d\left( W^{(k)} - \tilde{w}^{(k)}\right)_{\Tilde{s}}\right|.
    \end{align*}
    Now, using the Gronwall lemma in its integral form,
    \begin{align}\label{inequalityToapplyJensen}
        |I_{t_u} - \overline{I_{t_u}}| &\leq  \left| \int_0^{\dn} d\left( W^{(k)} - \tilde{w}^{(k)}\right)_{\Tilde{s}}\right| + \int_0^{t_u}  \left| \int_0^{\dn} d\left( W^{(k)} - \tilde{w}^{(k)}\right)_{\Tilde{s}} \right|L\exp(L(t_u - s))ds \nonumber\\
        &\leq  \left| \int_0^{\dn } d\left( W^{(k)} - \tilde{w}^{(k)}\right)_{\Tilde{s}}\right|\left( 1 + \int_0^{t_u}L\exp{(L(t_u-s))}ds\right) \nonumber\\
        &= \left| \int_0^{\dn} d\left( W^{(k)} - \tilde{w}^{(k)}\right)_{\Tilde{s}}\right|\exp(Lt_u),
    \end{align}
    where in the last line we have computed the integral explicitly.
    We observe that the martingale differences can be bounded in terms of $|I_{t_u} - \overline{I_{t_u}}|$, since according to \eqref{initialCIequation}, 
    \begin{equation*}
        |M_k - M_{k-1}|  \leq  \| F\|_{\text{Lip}}\left( \int_\Omega \sum_{u=1}^{n-k+1} | I_{t_u} - \overline{I_{t_u}} | \P_W(d\tilde{w})\right).
    \end{equation*}
    Combining $\eqref{inequalityToapplyJensen}$ for the case $u=1$ and $\eqref{ineq3}$ for the case $u>1$, we observe that, for all $p \geq 1$ we have that 
    \begin{align}\label{afterJensen}\nonumber
        |M_k - M_{k-1}|^p &\leq \| F\|_{\text{Lip}}^p 2^{p-1}\exp({Lp\dn})
        \left( \int_\Omega \left| \int_0^{\dn } d\left( W^{(k)} - \tilde{w}^{(k)}\right)_{\Tilde{s}}\right|\P_W(d\tilde{w}) \right)^p \\
        & + \| F\|_{\text{Lip}}^p 2^{p-1} \left( \sum_{u=2}^{n-k+1} \exp(-M(t_u-t_2))\right)^p \left( \int_\Omega  |I_{t_2} - \overline{I_{t_2}}|\P_W(d\tilde{w})\right)^p.
    \end{align}
    To simplify the notation, we  define the following expressions:
    \begin{equation*}
        \Psi_{n,k} := \sum_{u=2}^{n-k+1} \exp(-M(t_u-t_2)) \hspace{1.0cm }\text{ and } \hspace{1.0cm} G_{\dn}^{(k)} := \int_0^{\dn} d\left( W^{(k)} - \tilde{w}^{(k)}\right)_{\Tilde{s}}.
    \end{equation*}
    If we use the bound $\eqref{inequalityToapplyJensen}$ for $|I_{t_2} - \overline{I_{t_2}}|$ in  $\eqref{afterJensen}$, we can express the martingale differences as  
    \begin{equation*}
        |M_k - M_{k-1}|^p \leq 2^p\| F\|_{\text{Lip}}^p  \exp(Lpt_2) \Psi_{n,k}^p \left(\int_\Omega \left| G_{\dn}^{(k)} \right|\P_W(d\tilde{w}) \right)^p.
    \end{equation*}
    Now we proceed to find a bound for $\E[|M_k - M_{k-1}|^p| \mathcal{F}_{t_k - 1}]$ for $p \geq 2.$ If we denote  by $\{\mathcal{F}^{(k)}_t\}_{t \geq 0}$ the natural filtration associated to $W^{(k)}$,  since $W^{(k)}$ and $\Tilde{w}^{(k)}$ are independent from $\mathcal{F}_{t_{k-1}}$, we have that
    \begin{align}\label{ineq4}
        \E[|M_k - M_{k-1}|^p| \mathcal{F}_{t_{k - 1}}]  &\leq 
        2^p\| F\|_{\text{Lip}}^p \exp(Lpt_2) \Psi_{n,k}^p \E\left[\left(\int_\Omega  \left| G_{\dn}^{(k)}\right|\P_W(d\tilde{w})\right)^p\right] \nonumber\\
        &=
         2^p\| F\|_{\text{Lip}}^p  \exp(Lpt_2) \Psi_{n,k}^p
         \E\left[\E\left[\left| G_{\dn}^{(k)}\right|^p \Big|  \mathcal{F}^{(k)}_{t_1}\right]\right] \nonumber \\
         & = 2^p\| F\|_{\text{Lip}}^p  \exp(Lpt_2) \Psi_{n,k}^p \E\left[ \left| G_{\dn}^{(k)}\right|^p\right].
    \end{align}
       To control the conditional expectations of the increments of the martingale $M_k$, we  prove a bound for $\E[| G_{\dn}^{(k)}|^p ]$. We observe  that $G_{\dn}^{(k)} \overset{d}{=} \int_0^{\dn }dW_s$, which implies that $G_{\dn}^{(k)} \overset{d}{=} W_{\dn}$ being $(W_t)_{t \geq 0}$ a $d-$dimensional Brownian motion, and  $W_t^i$ the $i$th component of the process. 
    Since $-W_{\dn}^1 \overset{d}{=} W_{\dn}^1$, 
    \begin{align*}
        \P\left( \left| G_{\dn}^{(k)}\right| \geq x \right) &\leq d \left[ \P\left(  -\left(W_{\dn}^1\right) \geq \frac{x}{\sqrt{d}} \right) + \P\left(  W_{\dn}^1 \geq \frac{x}{\sqrt{d}} \right)\right] 
         \leq 2d \hspace{0.1cm}\P\left( W_{\dn}^1 \geq \frac{x}{\sqrt{d}} \right). 
    \end{align*}
   Using Markov's inequality and choosing $\eta =x/(\sqrt{d}\dn) $, we obtain that 
    \begin{align*}
        \P\left(\left| G_{\dn}^{(k)}\right| \geq x \right) &\leq 2d\E\left[\exp\left(\eta W_{\dn}^1\right)\right]\exp\left(-\eta xd^{-1/2}\right) \\ &= 2d\exp\left(\frac{\eta^2\dn}{2} - \frac{\eta x}{\sqrt{d}} \right) 
        =  2d\exp\left(-\frac{x^2}{2d \dn}\right).
    \end{align*}
    Now, using that $\E[X] = \int_0^{\infty} \P(X \geq x)dx$ for a non-negative random variable, and making the change of variables $x=s^p$, we have that 
    \begin{equation*}
        \E\left[\left| G_{\dn}^{(k)}\right|^p\right] \leq 2  d  p \int_{0}^\infty s^{p-1}\exp\left({-\frac{s^2}{2d\dn}}\right) ds.
    \end{equation*}
    If we proceed with the following change of variables $\tilde{s} = s^2/(2d\dn)$, we obtain that   
    \begin{equation}
        \E\left[ \left| G_{\dn}^{(k)}\right|^p\right] \leq dp\hspace{0.07cm}\Gamma\left(\frac{p}{2}\right)(2d\dn)^{\frac{p}{2}},
    \end{equation}
    where  $\Gamma(y) = \int_0^\infty \exp(-s)s^{y-1}ds.$
    From  $\eqref{ineq4}$, we have  that 
    \begin{equation*}
        \E[|M_k - M_{k-1}|^p| \mathcal{F}_{t_k - 1}]^{\frac{1}{p}} \leq 
        2\| F\|_{\text{Lip}}  \exp(2L\dn) \Psi_{n,k}\sqrt{d\dn}\left(2d\hspace{0.07cm}p\hspace{0.07cm}\Gamma\left(\frac{p}{2}\right) \right)^{\frac{1}{p}}.
    \end{equation*}
    Now, applying Lemma 3.1 of \cite{varvenne2019concentration} and using the fact that $\E[M_k - M_{k-1}|\mathcal{F}_{t_{k-1}}]=0$, we obtain that for any $\eta >0,$
    \begin{equation*}
        \E\left[\exp\left(\eta(M_k - M_{k-1})\right)|\mathcal{F}_{t_{k-1}}\right] \leq \exp{\left(16\eta^2 \| F\|_{\text{Lip}}^2\exp(4L\dn) \Psi_{n,k}^2 d^2 \dn  \right)}.
    \end{equation*}
    Now we observe that 
    \begin{align*}
        \E\left[\exp(\eta M_n)\right] &= \E\left[\exp(\eta M_{n-1})\E \left[ \exp(\eta (M_n - M_{n-1}))\right|\mathcal{F}_{t_{n-1}}]\right] \\
        & \leq  \E\left[\exp(\eta M_{n-1})\right]\exp{\left(16\eta^2 \| F\|_{\text{Lip}}^2\exp(4L\dn) \Psi_{n,n}^2 d^2 \dn  \right)}.
    \end{align*}
    Applying the same procedure recursively, we obtain that 
    \begin{align*}
          \E\left[\exp(\eta \left(F_Y - \E[F_Y]\right))\right] &=  \E\left[\exp(\eta \left(M_n - M_0\right))\right] \\&\leq \exp{\left(16\eta^2 \| F\|_{\text{Lip}}^2\exp(4L\dn)  d^2 \dn  \sum_{k=1}^n \Psi_{n,k}^2 \right)}.
    \end{align*}
    Since
    $$\Psi_{n,k} = \sum_{u=2}^{n-k+1} \exp(-M(t_u - t_2)) = \sum_{\tilde{u}=0}^{n-k-1}\exp(-\tilde{u}M \dn) \leq (1-\exp(-M \dn))^{-1},$$ then $\sum_{k=1}^n  \Psi_{n,k}^2 \leq n/(1-\exp(-M \dn))^2$ and  we conclude that
    \begin{equation*}
        \E\left[\exp{(\eta(F_Y - \E[F_Y]))}\right] \leq \exp{\left(\frac{16\eta^2 d^2 \| F\|_{\text{Lip}}^2 n \dn \exp(4L\dn) }{(1-\exp(-M \dn))^2}\right)}.
    \end{equation*}
    Now, if we apply Markov's inequality and set $\eta=\frac{r(1-\exp(-M \dn))^2}{32 n \dn d^2 \| F\|_{\text{Lip}}^2  \exp(4L\dn)}$, we obtain that 
    \begin{equation*}
        \P\left(F_Y - \E[F_Y] > r\right) \leq \exp{\left(\frac{-r^2 
        (1-\exp(-M \dn))^2 }{64d^2 \| F\|_{\text{Lip}}^2 n \dn \exp(4L\dn)}\right)}.
    \end{equation*}
    Finally, in case the functional $F$ is of the form $F_X=\frac{1}{n}\sum_{i=1}^n f(X_{t_i})$ for a Lipschitz function $f: \R^d \rightarrow \R$, we obtain the desired result.
\end{proof}

\subsubsection{Proof of the concentration inequality for the  OU case}
\begin{proof}
    For an element $v \in \R^d, \|v\|_2=1$, we define the centred Gaussian process $Y_t^v = v^\star X_t$ whose covariance kernel is given by $\E[Y_t^vY_s^v] = \rho_v(|t-s|)$, with $\rho_v(r) := v^\star \exp(-rA_0) C_\infty v$.
    We observe that $(Y_t^v)_{t \in [0,T]}$ can be considered as an isonormal Gaussian process indexed by a separable Hilbert space $\mathbb{H}$ whose scalar product is induced by the covariance kernel of $(Y_t^v)_{t \in [0,T]}$. In particular, we can write $Y_t^v = B(h_t)$ and $\langle h_t,h_s\rangle_{\mathbb{H}} = \rho(|t-s|)$. Now,  we define the quantity 
\begin{equation*}
Z_n^v := v^\star (C_T - C_\infty) v = \frac{1}{n} \sum_{i=1}^n \left((Y_{t_{i-1}}^v)^2 - \E[(Y_{t_{i-1}}^v)^2] \right).
\end{equation*}
We note that $Z_n^v$ is an element of the second order Wiener chaos. Hence, $Z_n^v$ has a Lebesgue density and condition $(ii)$ of Theorem $\ref{TheoremIvan}$ is satisfied. Moreover, since $L^{-1} Z_n^v = -(1/2)Z_n^v$, we have that 
\begin{gather*}
    \langle DZ_n^v, -DL^{-1} Z_n^v \rangle_{\mathbb{H}} = \frac{1}{2}\langle DZ_n^v, DZ_n^v \rangle_{\mathbb{H}} = \frac{2}{n^2} \sum_{i,j=1}^n \left\langle D(Y_{t_{i-1}}^v), D(Y_{t_{j-1}}^v) \right\rangle_{\mathbb{H}} \\
     \leq \frac{2}{n^2} \sum_{i,j=1}^n  |Y_{t_{i-1}}^vY_{t_{j-1}}^v| \left|\rho_v\left(t_{i-1} - t_{j-1}\right)\right| \leq \frac{2}{n^2}\sum_{i,j=1}^n |Y_{t_{i-1}}^v|^2 \left|\rho_v\left(t_{i-1} - t_{j-1}\right)\right|.
\end{gather*}
Adding and substracting $\E[ (Y_{t_{i-1}}^v)^2]$ we obtain 
\begin{gather*}
     \langle DZ_n^v, -DL^{-1} Z_n^v \rangle_{\mathbb{H}} \leq \frac{2}{n^2} \sum_{i,j=1}^n  \left(|Y_{t_{i-1}}^v|^2  - \E\left[ (Y_{t_{i-1}}^v)^2\right]\right)\left|\rho_v\left(t_{i-1} - t_{j-1}\right)\right|\\ + \frac{2}{n^2} \sum_{i,j=1}^n \E\left[ (Y_{t_{i-1}}^v)^2\right]\left|\rho_v\left(|t_{i-1} - t_{j-1}|\right)\right| =I_1 + I_2.
\end{gather*}
First, we work on $I_1$. Since $t_{i-1} = (i-1)\dn$ and the same applies to $t_{j-1}$, a change of variable leads us to  
\begin{gather*}
    I_1 \leq \frac{4}{n^2}\sum_{i=1}^n\sum_{k=0}^{n} \left(|Y_{t_{i-1}}^v|^2  - \E\left[ (Y_{t_{i-1}}^v)^2\right]\right)\left|\rho_v\left(\dn k\right)\right| 
    = \frac{4}{n}\sum_{k=0}^nZ_n^v |\rho_v(\dn k)|.
\end{gather*}
Similarly, expression $I_2$ can be controlled as
\begin{gather*}
    I_2 \leq \frac{4}{n^2}\sum_{i=1}^n\sum_{k=0}^{n}  \E\left[ (Y_{t_{i-1}}^v)^2\right]\left|\rho_v\left(\dn k\right)\right| 
    \leq \frac{4}{n}\sum_{k=0}^n\rho_v(0) |\rho_v(\dn k)|.
\end{gather*}
Using that $|\rho_{v}(\dn k )| \leq \| \exp(-\dn k A_0) \|_{\text{op}} \| C_{\infty} \|_{\text{op}} \leq \mathfrak{p}_0 \mathfrak{l}_{\max}\exp(-\mathfrak{m}\dn k) ,$ we observe that condition $(ii)$ of Theorem \ref{TheoremIvan}
is satisfied since
\begin{gather*}
     \langle DZ_n^v, -DL^{-1} Z_n^v \rangle_{\mathbb{H}} \leq \frac{4}{n}\sum_{k=0}^n |\rho_{v}(\dn k )|(Z_n^v + \rho_v(0)) \\\leq 
     \frac{4}{n} \mathfrak{p}_0 \mathfrak{l}_{\max} \sum_{k=0}^n \exp{(-\dn \mathfrak{m} k)}(Z_n^v + \mathfrak{l}_{\max}) \leq \alpha Z_n^v + \beta,
\end{gather*}
for $\alpha = 4n^{-1}\mathfrak{p}_0\mathfrak{l}_{\max}a$ and $\beta = 4n^{-1}\mathfrak{p}_0\mathfrak{l}_{\max}^2a$, where $a = (1 - \exp(-(n+1)\mathfrak{m}\dn))/(1  - \exp(-\mathfrak{m}\dn)).$ 
Therefore, result \eqref{CIOUbruta} follows.

When $\dn \rightarrow 0$, bounding from the infinite sum and applying a first order Taylor expansion, we observe that  
$\sum_{k=0}^n \exp{(-\dn \mathfrak{m} k)} \leq (1 - \exp(-\mathfrak{m}\dn))^{-1} \leq (\mathfrak{m}\dn)^{-1}$. Then, considering $a = (\mathfrak{m}\dn)^{-1}$ completes the proof. 
\end{proof}



\subsubsection*{Funding}
Chiara Amorino and Mark Podolskij were supported by the ERC Consolidator Grant 815703 “STAMFORD: Statistical Methods for High Dimensional Diffusions”. Francisco Pina's research is funded by the PRIDE Grant “MATHCODA: Mathematical Tools for Complex Data Structures”.



\bibliography{bibliografia}       


\end{document}